\documentclass[a4paper, 10pt]{amsart}

\usepackage{amsthm}
\usepackage{amsmath}
\usepackage{amsfonts}
\usepackage{amssymb}
\usepackage{latexsym}

\usepackage{graphicx}

\usepackage{tabularx}
\usepackage[margin=1.2in]{geometry}
\usepackage{enumitem}
\usepackage{ae}
\usepackage[T1]{fontenc}

\newtheorem{teo}{Theorem}
\newtheorem{cor}{Corollary}

\newtheorem{rem}{Remark}
\newtheorem{prop}{Proposition}
\newtheorem{lem}{Lemma}

\newcommand{\R}{\mathbb{R}}

\newcommand{\N}{\mathbb{N}}
\begin{document}
\title[Asympt. analysis for radial sign-changing solutions of the B.-N. problem]{Asymptotic analysis for  radial sign-changing solutions of the Brezis-Nirenberg problem}
\author{Alessandro Iacopetti}
\date{}
\thanks{Research partially supported by MIUR-PRIN project-201274FYK7\underline\ 005.}
\address[Alessandro Iacopetti]{Dipartimento di Matematica e Fisica, Universit\'a degli Studi di Roma Tre, L.go S. Leonardo Murialdo 1, 00146 Roma, Italy}
\email{iacopetti@mat.uniroma3.it}

\begin{abstract}
We study the asymptotic behavior, as $\lambda \rightarrow 0$, of least energy radial sign-changing solutions $u_\lambda$, of the Brezis-Nirenberg problem 

\begin{equation*}
\begin{cases}
-\Delta u = \lambda u + |u|^{2^* -2}u & \hbox{in}\ B_1\\
u=0 & \hbox{on}\ \partial B_1,
\end{cases}
\end{equation*}
where $\lambda >0$, $2^*=\frac{2n}{n-2}$ and $B_1$ is the unit ball of $\R^n$, $n\geq 7$.

We prove that both the positive and negative part $u_\lambda^+$ and $u_\lambda^-$ concentrate at the same point (which is the center) of the ball with different concentration speeds. 
Moreover we show that suitable rescalings of $u_\lambda^+$ and $u_\lambda^-$ converge to the unique positive regular solution of the critical exponent problem in $\R^n$.

Precise estimates of the blow-up rate of $\|u_\lambda^\pm\|_{\infty}$ are given, as well as asymptotic relations between $\|u_\lambda^\pm\|_{\infty}$ and the nodal radius $r_\lambda$.

Finally we prove that, up to constant,  $\lambda^{-\frac{n-2}{2n-8}} u_\lambda$ converges in $C_{loc}^1(B_1-\{0\})$ to $G(x,0)$, where $G(x,y)$ is the Green function of the Laplacian in the unit ball. 
\end{abstract}
\maketitle
\section{Introduction}
Let $n\geq 3$, $\lambda>0$ and $\Omega$ be a bounded open subset of $\R^n$ with smooth boundary. We consider the Brezis-Nirenberg problem 
\begin{equation}\label{PBN2}
\begin{cases}
-\Delta u = \lambda u + |u|^{2^* -2}u & \hbox{in}\ \Omega\\
u=0 & \hbox{on}\ \partial \Omega,
\end{cases}
\end{equation}
where $2^{*}=\frac{2n}{n-2}$ is the critical Sobolev exponent for the embedding of $H_0^1(\Omega)$ into $L^{2^*}(\Omega)$.
Problem (\ref{PBN2}) has been widely studied over the last decades, and many results for positive solutions have been obtained. 

The first existence result for positive solutions of (\ref{PBN2})  has been given by Brezis and Nirenberg in their classical paper \cite{2}, where, in particular the crucial role played by the dimension was enlightened. They proved that if $n\geq4$ there exist positive solutions of (\ref{PBN2}) for every $\lambda \in (0,\lambda_1(\Omega))$, where $\lambda_1(\Omega)$ denotes the first eigenvalue of $-\Delta$ on $\Omega$ with zero Dirichlet boundary condition. 
For the case $n=3$, which is more delicate, Brezis and Nirenberg \cite{2} proved that there exists $\lambda_*(\Omega)>0$ such that  positive solutions exist for every $\lambda \in (\lambda_*(\Omega),\lambda_1(\Omega))$.  When $\Omega=B$ is a ball they also proved that $\lambda_*(B)=\frac{\lambda_1(B)}{4}$ and 
a positive solution of (\ref{PBN2}) exists if and only if $\lambda \in (\frac{\lambda_1(B)}{4}, \lambda_1(B))$. Moreover, for more general bounded domains, they proved that if $\Omega \subset \R^3$ is strictly star-shaped about the origin there are no positive solutions for $\lambda$ close to zero.
We point out that weak solutions of (\ref{PBN2}) are classical solution. This is a consequence of a well known lemma of Brezis and Kato (see for instance Appendix B of \cite{STRUWE}).

The asymptotic behavior for $n\geq 4$, as $\lambda \rightarrow 0$, of positive solutions of (\ref{PBN2}), minimizing the Sobolev quotient, has been studied by Han \cite{HAN}, Rey \cite{Rey}. They showed, with different proofs, that such solutions blow up at exactly one point and they also determined the exact blow up rate as well as the location of the limit concentration points. 

Concerning the case of sign-changing solutions of (\ref{PBN2}), several existence results have been obtained if $n\geq 4$. In this case one can get sign-changing solutions for every $\lambda \in (0,\lambda_1(\Omega))$, or even $\lambda > \lambda_1(\Omega)$, as shown in the papers of Atkinson-Brezis-Peletier \cite{ABP2}, Clapp-Weth \cite{CW}, Capozzi-Fortunato-Palmieri \cite{CFP}. The case $n=3$ presents the same difficulties enlightened before for positive solutions and even more. In fact, differently from the case of positive solutions, it is not yet known, when $\Omega=B$ is a ball in $\R^3$, if there are sign-changing solutions of (\ref{PBN2}) when $\lambda$ is smaller than $\lambda_*(B)=\lambda_1(B)/4$. A partial answer to this question posed by H. Brezis has been given in \cite{AMP3}. 

The blow-up analysis of low-energy sign-changing solutions of (\ref{PBN2}) has been done by Ben Ayed-El Mehdi-Pacella \cite{AMP3},\cite{AMP7}. In \cite{AMP3} the authors analyze the case $n=3$. They introduce the number defined by
$$\bar\lambda(\Omega):= \inf \{\lambda \in \R^+; \ \hbox{Problem (\ref{PBN2}) has a sign-changing solution}\ u_\lambda, \hbox{with} \ \|u_\lambda\|_\Omega^2 - \lambda |u_\lambda|_{2,\Omega}^2 \leq 2S^{3/2}  \},$$
where $\|u_\lambda\|_\Omega^2 = \int_\Omega |\nabla u_\lambda|^2 \ dx$, $|u_\lambda|_{2,\Omega}^2 = \int_\Omega |u_\lambda|^2 \ dx$ and $S$ is the best Sobolev constant for the embedding $H_0^{1}(\Omega)$ into $L^{2*}(\Omega)$. To be precise they study the behavior of sign-changing solutions of (\ref{PBN2}) which converge weakly to zero and whose energy converges to $2S^{3/2}$ as $\lambda \rightarrow \bar \lambda(\Omega)$. They prove that these solutions blow up at two different points $\bar a_1$, $\bar a_2$, which are the limit of the concentration points $a_{\lambda,1}$, $a_{\lambda,2}$ of the positive and negative part of the solutions. Moreover the distance between $a_{\lambda,1}$ and $a_{\lambda,2}$ is bounded from below by a positive constant depending only on $\Omega$ and the concentration speeds of the positive and negative parts are comparable. This result shows that, in dimension 3, there cannot exist, in any bounded smooth domain $\Omega$, sign-changing low energy solutions whose positive and negative part concentrate at the same point.

In higher dimensions ($n\geq 4$), the same authors, in their paper \cite{AMP7}, describe the asymptotic behavior, as $\lambda \rightarrow 0$, of sign-changing solutions of (\ref{PBN2}) whose energy converges to the value $2S^{n/2}$. 
Even in this case they prove that the solutions concentrate at two separate points, but, they need to assume the extra hypothesis that the concentration speeds of the two concentration points are comparable, while in dimension three this was derived without any extra assumption (see Theorem 4.1 in \cite{AMP3}). They also describe in \cite{AMP7} the asymptotic behavior, as $\lambda \rightarrow 0$, of the solutions outside the limit concentration points proving that there exist positive constants $m_1, m_2$ such that
$$\lambda^{-\frac{n-2}{2n-8}}u_\lambda \rightarrow m_1 G(x,\bar a_1)-m_2 G(x,\bar a_2) \ \ \hbox{in} \ C_{loc}^2(\Omega - \{\bar a_1, \bar a_2\}), \ \hbox{if } n\geq 5,$$
$$\|u_\lambda\|_{\infty}u_\lambda \rightarrow m_1 G(x,\bar a_1)-m_2 G(x,\bar a_2) \ \ \hbox{in} \ C_{loc}^2(\Omega - \{\bar a_1, \bar a_2\}), \ \hbox{if } n= 4,$$
where $G(x,y)$ is the Green's function of the Laplace operator in $\Omega$.
So for $n\geq 4$ the question of proving the existence of sign-changing low-energy solutions (i.e. such that $\|u_\lambda\|_\Omega^2$ converges to $2S^{n/2}$ as $\lambda \rightarrow 0$) whose positive and negative part concentrate at the same point, was left open.

To the aim to contribute to this question as well as to describe the precise asymptotic behavior of radial sign-changing solutions we consider the Brezis-Nirenberg problem in the unit ball $B_1$, i.e.
 \begin{equation}\label{PBN}
\begin{cases}
-\Delta u = \lambda u + |u|^{2^* -2}u & \hbox{in}\ B_1\\
u=0 & \hbox{on}\ \partial B_1.
\end{cases}
\end{equation}

It is important to recall that Atkinson-Brezis-Peletier \cite{ABP}, Adimurthi-Yadava \cite{AY2}  showed, with different proofs, that for $n=3,4,5,6$ there exists $\lambda^*=\lambda^*(n)>0$ such that there is no radial sign-changing solution of (\ref{PBN}) for $\lambda \in (0,\lambda^*)$. Instead they do exist if $n\geq 7$, as shown by Cerami-Solimini-Struwe in their paper \cite{1}. In Proposition \ref{antani} (see also Remark \ref{rmk1}) we recall this existence result and get the limit energy of such solutions as $\lambda \rightarrow 0$.

In view of these results we analyze the case $n\geq7$ and $\lambda \rightarrow 0$. More precisely we consider a family $(u_\lambda)$ of least energy sign-changing solutions of (\ref{PBN}). It's easy to see that $u_\lambda$ has exactly two nodal regions. We denote by $r_\lambda \in (0,1)$ the node of $u_\lambda=u_\lambda(r)$ and, without loss of generality, we assume $u_\lambda(0)>0$, so that $u_\lambda^+$ is different from zero in $B_{r_\lambda}$ and $u_\lambda^-$ is different from zero in the annulus $A_{r_\lambda}:=\{x \in \R^n; \ r_\lambda < |x| <1\}$.  

 We set $M_{\lambda, +}:=\|u_\lambda^+\|_{\infty}$, $M_{\lambda, -}:=\|u_\lambda^-\|_{\infty}$, $\beta:=\frac{2}{n-2}$, $\sigma_\lambda:=M_{\lambda, +}^\beta r_\lambda$,  $\rho_\lambda:=M_{\lambda, -}^\beta r_\lambda$. Moreover, for $\mu >0$, $x_0 \in \R^n$, let $\delta_{x_0,\mu}$ be the function $\delta_{x_0,\mu}:\R^n \rightarrow \R$ defined by
\begin{equation}\label{stndbubble}
\delta_{x_0,\mu}(x):=\frac{[n(n-2)\mu^2]^{(n-2)/4}}{[\mu^2+|x-x_0|^2]^{(n-2)/2}}.
\end{equation}
Proposition \ref{PROP1} states that both $M_{\lambda, +}$ and $M_{\lambda, -}$ diverge, $u_\lambda$ weakly converge to $0$ and $\|u_\lambda^\pm\|_{B_1}^2 \rightarrow S^{n/2}$, as $\lambda \rightarrow 0$. The results  of this paper are contained in the following theorems.

\begin{teo}\label{mainteo}
Let  $n\geq 7$ and $(u_\lambda)$ be a family of least energy  radial sign-changing solutions of (\ref{PBN}) and $u_\lambda(0)>0$.
Consider the rescaled functions $\tilde u_\lambda^+(y):=\frac{1}{M_{\lambda, +}}u_\lambda^+\left(\frac{y}{M_{\lambda, +}^\beta}\right)$  in $B_{\sigma_\lambda}$, and $\tilde u_\lambda^-(y):=\frac{1}{M_{\lambda, -}}u_\lambda^-\left(\frac{y}{M_{\lambda, -}^\beta}\right)$  in $A_{\rho_\lambda}$, where $B_{\sigma_\lambda}:=M_{\lambda, +}^\beta B_{r_\lambda}$, $A_{\rho_\lambda}:=M_{\lambda, -}^\beta A_{r_\lambda}$. Then:
\begin{description}
\item[(i)] $\tilde u_\lambda^+ \rightarrow \delta_{0,\mu}$ in  $C_{loc}^2(\R^n)$ as $\lambda \rightarrow 0$, where $\delta_{0,\mu}$ is the  function defined in (\ref{stndbubble}) for $\mu=\sqrt{n(n-2)}$.
\item[(ii)] $\tilde u_\lambda^- \rightarrow \delta_{0,\mu}$ in $C_{loc}^2(\R^n-\{0\})$ as $\lambda \rightarrow 0$, where $\delta_{0,\mu}$ is the same as in (i).
\end{description}
\end{teo}

From this theorem we deduce that the positive and negative parts of $u_\lambda$ concentrate at the origin. Moreover, as a consequence of the preliminary results for the proof of Theorem \ref{mainteo}, we show that $M_{\lambda, +}$ and $M_{\lambda, -}$ are not comparable, i.e. $\frac{M_{\lambda, +}}{M_{\lambda, -}} \rightarrow +\infty$ as $\lambda \rightarrow 0$. Indeed we are able to determine the exact rate of $M_{\lambda, -}$ and an asymptotic relation between $M_{\lambda, +}$, $M_{\lambda, -}$ and the radius $r_\lambda$.

\begin{teo}\label{asestpropmain}
As $\lambda\rightarrow 0$ we have the following:
\begin{description}
\item[(i)] $M_{\lambda,+}^{2-2\beta} r_\lambda^{n-2}\lambda \rightarrow c(n)$,
\item[(ii)] $M_{\lambda,-}^{2-2\beta} \lambda \rightarrow c(n)$,
\item[(iii)] $\displaystyle \frac{M_{\lambda,-}^{2-2\beta}}{M_{\lambda,+}^{2-2\beta}r_\lambda^{n-2}} \rightarrow 1$,
\end{description}
where $c(n):=\frac{c_1^2(n)}{c_2(n)}$, $c_1(n):=\int_{0}^{\infty} \delta_{0,\mu}^{2^*-1}(s) s^{n-1}ds$, $c_2(n):=2\int_{0}^{\infty} \delta_{0,\mu}^{2}(s) s^{n-1}ds$, $\mu=\sqrt{n(n-2)}$.
\end{teo}

The last result we provide is about the asymptotic behavior of the functions $u_\lambda$ in the ball $B_1$, outside the origin. We show that, up to a constant, $\lambda^{- \frac{n-2}{2n-8}} u_\lambda$ converges in $C_{loc}^1(B_1-\{0\})$ to $G(x,0)$, where $G(x,y)$  is the Green function of the Laplace operator in $B_1$. 

\begin{teo}\label{greentea}
As  $\lambda \rightarrow 0$ we have $$\lambda^{-\frac{n-2}{2n-8}} u_\lambda \rightarrow \tilde c(n) G(x,0) \ \ \hbox{ in}\ C_{loc}^1(B_1-\{0\}),$$  
where $G(x,y)$ is the Green function for the Laplacian in the unit ball, $\tilde c(n)$ is the constant defined by
$\tilde c(n):=\omega_n \frac{c_2(n)^{\frac{n-2}{2n-8}}}{c_1(n)^{\frac{4}{2n-8}}} $, $\omega_n$ is the measure of the $(n-1)$-dimensional unit sphere $S^{n-1}$ and $c_1(n), c_2(n)$ are the constants appearing in Theorem  \ref{asestpropmain}.
\end{teo}

The proof of the above results are technically complicated and often rely on the radial character of the problem. We would like to stress that the presence of the lower order term $\lambda u$ makes our analysis quite different from that performed in \cite{3b} for low energy sign-changing solution of an almost critical problem.

Since we consider nodal solutions our results cannot be obtained by following the proofs for the case of positive solutions (\cite{HAN}, \cite{Rey}, \cite{10}, \cite{11}). In particular, in order to analyze the behavior of the negative part $u_\lambda^-$, which is defined in an annulus, we prove a new uniform estimate (Proposition \ref{lbnp} and Proposition \ref{rescpn}) which holds for any dimension $n\geq 3$ and is of its own interest (see Remark \ref{uupbld} and Proposition \ref{lbnpld}).

For the sake of completeness let us mention that our results, as well as those of \cite{3b}, show a big difference between the asymptotic behavior of radial sign-changing solutions in dimension $n>2$ and $n=2$. Indeed, in this last case, the limit problems as well as the limit energies of the positive and negative part of solutions are different (see \cite{GP}).

Finally we point out that, in view of the above theorems it is natural to ask whether solutions of (\ref{PBN2}) which behave like the radial ones exist in other bounded domains. More precisely, it would be interesting to show the existence of sign-changing solutions whose positive and negative part concentrate at the same point but with different speeds, each one carrying the same energy.

In the paper in preparation \cite{IACVAIR} we answer positively this question at least in the case of some symmetric domains in $\R^n$, $n\geq 7$.

We conclude observing that this type of bubble tower solutions have interest also for the associated parabolic problem, since, as proved \cite{MPS}, \cite{CDW}, \cite{DPS}, they induce a peculiar blow-up phenomenon for the initial data close to them.

The paper is divided into 6 sections. In the second one we give some preliminary results on radial sign-changing solutions. In Section 3 we prove estimates for solutions with two nodal regions and, in particular, prove the new uniform estimate of Proposition \ref{rescpn}.

In Section 4 we analyze the asymptotic behavior of the rescaled solutions and prove Theorem \ref{mainteo}.
Section 5 is devoted to the study of the divergence rate of $\|u_\lambda^\pm\|_\infty$, as $\lambda \rightarrow 0$ and to the proof of Theorem \ref{asestpropmain}. Finally in Section 6 we prove Theorem \ref{greentea}.

\section{Preliminary results on radial sign-changing  solutions}
In this section we recall or prove some results about the existence and qualitative properties of radial sign-changing solutions of the Brezis-Nirenberg problem (\ref{PBN}).

We start with the following:

\begin{prop}\label{antani}
Let $n\geq 7$,  $k \in \N^+$ and $\lambda \in (0,\lambda_1)$, where $\lambda_1$ is the first eigenvalue of $-\Delta$ in $H_0^1(B_1)$. Then there exists  a radial sign-changing solution $u_{k,\lambda}$  of (\ref{PBN}) with the following properties:
\begin{description}
\item[(i)] $u_{k,\lambda}(0) >0$,
\item[(ii)] $u_{k,\lambda}$ has exactly k nodal regions in $B_1$,
\item[(iii)] $I_\lambda(u_{k,\lambda})= \frac{1}{2} \left( \int_{B_1}|\nabla u_{k,\lambda}|^2 - \lambda |u_{k,\lambda}|^2 \ dx\right)-\frac{1}{2^*} \int_{B_1} |u_{k,\lambda}|^{2^*} \ dx \rightarrow \frac{k}{n} S^{n/2}$
as $\lambda \rightarrow 0$, where $S$ is the best constant for the Sobolev embedding $H_0^1(B_1) \hookrightarrow L^{2^*}(B_1)$.
\end{description}
\end{prop}
\begin{proof}
The existence of radial solutions of (\ref{PBN}) satisfying (i) and (ii) is proved in \cite{1}. 
It remains only to prove (iii). To do this we need to introduce some notations and recall some facts proved in \cite{1} and \cite{2}. Let $k \in \N^+$ and $0=r_0 <r_1 < \ldots < r_k = 1$ any partition of the interval $[0,1]$, we define the sets $\Omega_1:=B_{r_1}=\{x \in B_1; |x|<r_1\}$ and, if $k\geq 2$, $\Omega_j:=\{x \in B_1; r_{j-1}<|x|<r_{j}\}$ for $j=2,\ldots,k$. 

Then we consider the set
\begin{eqnarray*}
\mathcal{M}_{k,\lambda}:=&\displaystyle \left\{u \in H_{0,rad}^1(B_1); \ \hbox{there exists a partition} \ 0=r_0 <r_1 < \ldots < r_k = 1\right. \\ & \ \ \ \ \ \ \hbox{such that:} \ u(r_j)=0, \hbox{for}\  1\leq j \leq k,\  (-1)^{j-1}u(x) \geq 0 , u \not\equiv 0 \ \hbox{in} \ \Omega_j, \ \hbox{and} \\
&\left. \displaystyle \int_{\Omega_j} \left(|\nabla u_j|^2 - u_j^2 - |u_j|^{2^*} \right)\ dx=0,\ \hbox{for}\  1\leq j \leq k \right\}, 
\end{eqnarray*}
where $H_{0,rad}^1(B_1)$ is the subspace of the radial functions in $H_0^1(B_1)$ and $u_j$ is the function defined by $u_j:= u\ \chi_{_{\Omega_j}}$, where $\chi_{_{\Omega_j}}$ denotes the characteristic function of $\Omega_j$. Note that for any $k \in \N^+$ we have $\mathcal{M}_{k,\lambda}\neq \emptyset$, so we define
$$c_k(\lambda):=\inf_{\mathcal{M}_{k,\lambda}} I_\lambda (u).$$

 In \cite{1} the authors prove, by induction on $k$, that for every $k \in \N^+$ there exists $u_{k,\lambda} \in \mathcal{M}_{k,\lambda} $ such that $I_\lambda(u_{k,\lambda})=c_k(\lambda)$ and $u_{k,\lambda}$ solves (\ref{PBN}) in $B_1$.  Moreover they prove that 
\begin{equation}\label{sbirigunda}
c_{k+1}(\lambda) < c_k(\lambda) + \frac{1}{n} S^{n/2}.
\end{equation}

Note that for $k=1$ $u_{1,\lambda}$ is just the positive solution found in \cite{2}, since, by the Gidas, Ni and Nirenberg symmetry result \cite{5} every positive solution is radial and from  \cite{6} or \cite{9} we know that positive solutions of (\ref{PBN}) are unique.

To prove (iii) we argue by indution. Since $c_1(0)=\frac{1}{n}S^{n/2}$, by continuity we get that $c_1(\lambda)\rightarrow \frac{1}{n}S^{n/2}$, as $\lambda \rightarrow 0$, so that (iii) holds for $k=1$.

Now assume that  $c_k(\lambda) \rightarrow  \frac{k}{n} S^{n/2}$, and let us to prove that
$c_{k+1}(\lambda)=I_\lambda(u_{k+1,\lambda}) \rightarrow  \frac{k+1}{n} S^{n/2}$. 

Let us observe that $c_{k+1}(\lambda) \geq (k+1) c_1(\lambda)$. In fact  $w:=u_{k+1,\lambda}$ achieves the minimum for $I_\lambda$ over ${\mathcal{M}_{k+1,\lambda}}$, so that, by definition, it has $k+1$ nodal regions and $w_j:= w \chi_{_{\Omega_j}}$ belongs to $H_{0,rad}^1(B_1)$  for all $j=1,\ldots,k+1$. Since $w \in \mathcal{M}_{k+1,\lambda} $ we have, depending on the parity of $j$, that one between $ w_j^+$ and $w_j^-$ is not zero and belongs to $ \mathcal{M}_{1,\lambda}$, we denote it by $\tilde w_j$.
Then $I_\lambda (\tilde w_j) \geq c_1(\lambda)$ for all $j=1,\ldots,k+1$ and hence 
$$c_{k+1}(\lambda)=I_\lambda(w)=\sum_{j=1}^{k+1}I_\lambda (w_j^{\pm}) \geq (k+1) c_1(\lambda).$$

Combining this with (\ref{sbirigunda}) we get 
$$c_k(\lambda) + \frac{1}{n} S^{n/2} > c_{k+1}(\lambda)\geq (k+1) c_1(\lambda).$$
Since by induction hypothesis $c_k(\lambda) \rightarrow  \frac{k}{n} S^{n/2}$ as $\lambda \rightarrow 0$ and  we have proved that $c_1(\lambda) \rightarrow \frac{1}{n} S^{n/2}$ we get that $c_{k+1}(\lambda)\rightarrow \frac{k+1}{n} S^{n/2}$, and the proof is concluded.
\end{proof}

\begin{rem}\label{rmk1}
Let  $k \in \N^+$ and $(u_\lambda)$ be a family of solutions of (\ref{PBN}), satisfying (iii) of Proposition \ref{antani}, then $\|u_\lambda\|_{B_1}^2=\int_{B_1}|\nabla u_\lambda|^2\ dx \rightarrow kS^{n/2}$, as $\lambda \rightarrow 0$. 
\end{rem}
This comes easily from Proposition \ref{antani}, and the fact that $u_\lambda$ belongs to the Nehari manifold $\mathcal{N}_\lambda$ associated to (\ref{PBN}), which is defined by $$\mathcal{N}_\lambda:=\{u \in H_0^1(B_1); \ \|u\|_{B_1}^2 - \lambda |u|_{2,B_1}^2=|u|_{2^*,B_1}^{2^*}\}.$$

The first qualitative property we state about any radial sign-changing solution $u_\lambda$ of (\ref{PBN}) is that the global maximum point of $|u_\lambda|$ is located at the origin; which is a well-known fact for positive solutions of (\ref{PBN}), as consequence of \cite{5}. 

\begin{prop} \label{TROLL}
Let $u_\lambda$ be a radial solution of (\ref{PBN}), then we have $|u_\lambda (0)|=\|u_\lambda\|_{\infty}$.
\end{prop}
\begin{proof}
Since $u_\lambda=u_\lambda(r)$ is a radial solution of (\ref{PBN}), then it solves
\begin{equation}\label{RODE}
\begin{cases}
u_\lambda^{\prime \prime} +\frac{n-1}{r}u_\lambda^{\prime} + \lambda u_\lambda + |u_\lambda |^{2^* -2}u_\lambda =0  & \hbox{in}\ \ (0,1)\\
u_\lambda^\prime( 0)=0, \ \ u_\lambda(1)=0.&
\end{cases}
\end{equation}
Multiplying the equation by $u_\lambda^\prime$ we get
$$u_\lambda^{\prime \prime} u_\lambda^\prime + \lambda u_\lambda u_\lambda^\prime+ |u_\lambda |^{2^* -2}u_\lambda u_\lambda^\prime= - \frac{n-1}{r}(u_\lambda^{\prime})^2 \leq 0. $$
We rewrite this as
$$\frac{d}{dr}\left[\frac{ (u_\lambda^\prime)^2 }{2}+ \lambda \frac{u_\lambda^2} {2}+ \frac{|u_\lambda |^{2^*}}{2^*}\right] \leq 0. $$
Which implies that the function
 $$E(r):= \frac{ (u_\lambda^\prime)^2 }{2}+ \lambda \frac{u_\lambda^2} {2}+ \frac{|u_\lambda |^{2^*}}{2^*}$$ is not increasing. So $E(0) \geq E(r)$ for all $r \in (0,1)$, where $E(0)=\lambda \frac{(u_\lambda(0))^2} {2}+ \frac{|u_\lambda(0) |^{2^*}}{2^*}$. Assume that $r_0 \in (0,1)$ is the global maximum for $|u_\lambda|$, so we have $u_\lambda^\prime(r_0)=0$, $|u_\lambda(r_0)|=\|u_\lambda\|_{\infty}$ and $E(r_0)=\lambda \frac{\|u_\lambda\|_{\infty}^2} {2}+ \frac{\|u_\lambda \|_{\infty}^{2^*}}{2^*}$. 

 Now we observe that, for all $\lambda >0$, the function $g(x):=\frac{\lambda}{2}x^2 + \frac{1}{2^*}x^{2^*}$, defined in $\R^+ \cup \{0\}$, is strictly increasing; thus we have $E(r_0) \geq E(0)$ and hence $E(r_0)=E(0)$. Since $g$ is strictly increasing we get $|u_\lambda(0)|=|u_\lambda(r_0)|=\|u_\lambda\|_{\infty}$ and we are done. 
\end{proof}

A consequence of the previous proposition is the following:

\begin{cor}\label{remutile}
Assume $u_\lambda$ is a nontrivial radial solution of (\ref{PBN}). If $0\leq r_1\leq r_2<1$ are two points in the same nodal region  such that $|u_\lambda(r_1)|\leq |u_\lambda(r_2)|$, $u_\lambda^\prime(r_1)=u_\lambda^\prime(r_2)=0$, then necessarily $r_1=r_2$. 
\end{cor}
\begin{proof}
Assume by contradiction $r_1<r_2$. By the assumptions and since the function $g(x):=\frac{\lambda}{2}x^2 + \frac{1}{2^*}x^{2^*}$ is a strictly increasing function (in $\R^+ \cup \{0\}$), we have $E(r_1)=g(|u_\lambda(r_1)|)\leq g(|u_\lambda(r_2)|)=E(r_2)$. But, as proved in Proposition \ref{TROLL}, $E(r)$ is a decreasing function, so necessarily $E(r_1)=g(|u_\lambda(r_1)|)= g(|u_\lambda(r_2)|)=E(r_2)$ from which we get $|u_\lambda(r_1)|=|u_\lambda(r_2)|$.  Since $r_1,r_2$ are in the same nodal region from $|u_\lambda(r_1)|=|u_\lambda(r_2)|$ we have $u_\lambda(r_1)=u_\lambda(r_2)$, thus there exists $r_* \in (r_1,r_2)$ such that $u_\lambda^\prime(r_*)=0$, and, since $E(r)$ is a decreasing function, we have $E(r_1)\geq E(r_*) \geq E(r_2)$. 
From this we deduce $  g(|u_\lambda(r_1)|)\geq g(|u_\lambda(r_*)|)\geq g(|u_\lambda(r_2)|)$, and hence $u_\lambda(r_1) = u_\lambda(r_*) = u_\lambda(r_2)$. 
Therefore $u_\lambda$ must be constant in the interval $[r_1,r_2]$ and, being a solution of (\ref{PBN}), it must be zero in that interval. In fact, since (\ref{PBN}) is invariant under a change of sign, we can assume that $u_\lambda \equiv c >0$. Then, by the strong maximum principle, $u_\lambda$ must be zero in the nodal region to which $r_1$, $r_2$ belong. This, in turn, implies that $u_\lambda$ is a trivial solution of (\ref{PBN}) which is a contradiction.
\end{proof}

\section{Asymptotic results for solutions with 2 nodal regions}
\subsection{General results}
Let $(u_\lambda)$ be a family of least energy radial, sign-changing solutions of (\ref{PBN}) and such that $u_\lambda(0)>0$.  

We denote by $r_\lambda \in (0,1)$ the node; so we have $u_\lambda >0$ in the ball $B_{r_\lambda}$ and $u_\lambda<0$ in the annulus $A_{r_\lambda}:=\{x \in \R^n ; r_\lambda <|x| < 1\}$. We write $u_\lambda^\pm$ to indicate that the statements  hold both for the positive and negative part of $u_\lambda$.

\begin{prop}\label{PROP1}
We have:
\begin{description}
\item[(i)] $\|u_\lambda^\pm\|_{B_1}^2  = \int_{B_1}|\nabla u_\lambda^\pm|^2\  dx \rightarrow S^{n/2}$, as $\lambda \rightarrow 0$,
\item[(ii)] $|u_\lambda^\pm|_{2^*,B_1}^{2^*}  = \int_{B_1}| u_\lambda^\pm|^{\frac{2n}{n-2}}\  dx \rightarrow S^{n/2}$, as $\lambda \rightarrow 0$,
\item[(iii)] $u_\lambda \rightharpoonup 0$, as $\lambda \rightarrow 0$,
\item[(iv)] $\displaystyle M_{\lambda,+}:=\max_{B_1}u_\lambda^+ \rightarrow + \infty$, $\displaystyle M_{\lambda,-}:=\max_{B_1}u_\lambda^- \rightarrow + \infty$, as $\lambda \rightarrow 0.$
\end{description}
\end{prop}
\begin{proof}
This proposition is a special case of Lemma 2.1 in \cite{AMP7}.
\end{proof}

Let's recall a classical result, due to Strauss, known as "radial lemma":
\begin{lem}[Strauss]\label{strauss}
 There exists a constant $c >0$, depending only on $n$, such that for all $u \in H_{rad}^1(\R^n)$
\begin{equation} \label{eqlevistrauss}
|u(x)| \leq c\ \frac{\|u\|_{1,2}^{1/2}}{|x|^{(n-1)/2}} \ \ \hbox{a.e. on}\  \R^n,
\end{equation}
where $\|\cdot\|_{1,2}$ is the standard $H^1$-norm.
\end{lem}
\begin{proof}
For the proof of this result see for instance \cite{WLM}.
\end{proof}

We denote by $s_\lambda \in (0,1)$ the global minimum point of $u_\lambda=u_\lambda(r)$, so we have $0<r_\lambda < s_\lambda$, $u_\lambda^-(s_\lambda)=M_{\lambda,-}$. The following proposition gives an information
on the behavior of $r_\lambda$ and $s_\lambda$ as $\lambda \rightarrow 0$.

\begin{prop}\label{LUS}
We have  $s_\lambda \rightarrow 0$ (and so $r_\lambda \rightarrow 0$ as well), as $\lambda \rightarrow 0$. 
\end{prop}

\begin{proof}
Assume by contradiction that $s_{\lambda_m} \geq s_0$ for a sequence $\lambda_m\rightarrow 0$ and for some $0<s_0<1$. Then by Lemma \ref{strauss} we get
$$M_{\lambda_m,-}=|u_{\lambda_m}(s_{\lambda_m})| \leq c \frac{\|u_{\lambda_m}\|_{1,2,B_1}^{1/2}}{s_{\lambda_m}^{(n-1)/2}} \leq c \frac{\|u_{\lambda_m}\|_{1,2,B_1}^{1/2}}{s_0^{(n-1)/2}},$$
where $c$ is a positive constant depending only on $n$.
Since $|\nabla u_\lambda|_{2,B_1}^2 \rightarrow 2S^{n/2}$ as $\lambda \rightarrow 0$ it follows that $M_{\lambda_m,-}$ is bounded, which is a contradiction.
\end{proof}
We recall another well known proposition:

\begin{prop}\label{UNLP}
Let $u \in C^2(\R^n)$ be a solution of
 \begin{equation}\label{WSLPBN}
\begin{cases}
-\Delta  u =   |u|^{2^* -2}u & \hbox{in}\ \R^n\\
  u\rightarrow 0 &\ \hbox{as} \ |y|\rightarrow +\infty.
\end{cases}
\end{equation} 
Assume that $u$ has a finite energy $I_0(u):=\frac{1}{2}|\nabla u|_{2,\R^n}^2-\frac{1}{2^*}|u|_{2^*,\R^n}^{2^*}$ and  $u$ satisfies one of these assumptions:
\begin{description}
\item[(i)] u is positive (negative) in $\R^n$,
\item[(ii)] u is spherically symmetric about some point.
\end{description}
Then there exist $\mu >0$, $x_0 \in \R^n$ such that $u$ is one of the functions $\delta_{x_0,\mu}$, defined in (\ref{stndbubble}).
\end{prop}
\begin{proof}
A sketch of the proof can be found in \cite{1}, Proposition 2.2.
\end{proof}

\subsection{An upper bound for $u_\lambda^{+}$, $u_\lambda^{-}$}
In this section we recall an estimate for positive solutions of (\ref{PBN}) in a ball and we generalize it to get an upper bound for  $u_\lambda^{-}$, which is defined in the annulus $A_{r_\lambda}:=\{x \in \R^n ; r_\lambda <|x| < 1\}$. 
\begin{prop}\label{lbpp}
Let $n\geq 3$ and $u$ be a solution of
 \begin{equation}\label{LAP}
\begin{cases}
-\Delta  u =  \lambda u + u^{\frac{n+2}{n-2}} & \hbox{in}\ \ B_R\\
u>0 &  \hbox{in}\ \ B_R\\
  u=0 &\ \hbox{in} \ \hbox{on}\ \ \partial B_R,
\end{cases}
\end{equation} 
for some positive $\lambda$.
 Then $u(x) \leq w(x,u(0))$ in $B_R$, where
$$w(x,c):= c \left\{1+\frac{c^{-1}f(c)}{n (n-2)}|x|^2\right\}^{-(n-2)/2},$$
and  $f:[0,+\infty) \rightarrow [0,+\infty)$ is the function defined by $f(y):=\lambda y + y^{\frac{n+2}{n-2}}$.
\end{prop}
\begin{proof}
The proof is based on the results contained in the papers of Atkinson and Peletier \cite{10}, \cite{11}. Since the solutions of (\ref{LAP}) are radial (see \cite{5}) we consider the ordinary differential equation associated to (\ref{LAP}) which, by some change of variable, can be turned into an Emden-Fowler equation. For it is easy to get the desired upper bound. All details are given in the next Proposition \ref{lbnp}.
\end{proof}

\begin{rem}\label{efrem1}
The previous proposition gives an upper bound for $u_\lambda^+$. In fact, taking into account that $u_\lambda^+$ is defined and positive in  the ball $B_{r_\lambda}$ and $u_\lambda^+(0)=M_{\lambda,+}$, we have
\begin{equation}\label{lbppeq}
\begin{array}{lll}
u_\lambda^+(x) &\leq&\displaystyle M_{\lambda,+} \left\{1+\frac{M_{\lambda,+}^{-1}\ f(M_{\lambda,+})}{n (n-2)}|x|^2\right\}^{-(n-2)/2}\\[16pt]
&=& \displaystyle M_{\lambda,+} \left\{1+\frac{\lambda  + M_{\lambda,+}^{\frac{4}{n-2}}}{n (n-2)}|x|^2\right\}^{-(n-2)/2},
\end{array}
\end{equation}
for all $x \in B_{r_\lambda}$.
\end{rem}

\begin{prop}\label{lbnp}
Let $u_\lambda$ be as in Section 3.1 and $\epsilon \in (0, \frac{n-2}{2})$.  There exist $\delta=\delta(\epsilon) \in (0,1)$, $\delta(\epsilon) \rightarrow 1$ as $\epsilon \rightarrow 0$ and a positive constant $\overline{\lambda}=\overline{\lambda}(\epsilon)$, such that for all $\lambda \in (0,\overline{\lambda})$ we have
\begin{equation}\label{stimapp}
u_\lambda^-(x) \leq M_{\lambda,-} \left\{1+\frac{M_{\lambda,-}^{-1}\ f(M_{\lambda,-})}{n (n-2)} c(\epsilon)|x|^2\right\}^{-(n-2)/2}, 
\end{equation}
for all $x \in A_{\delta, \lambda}$, where $A_{\delta, \lambda}:= \{ x \in \R^n; \ \delta^{- 1/N}s_\lambda <|x|<1\}$,  
 $c(\epsilon)=\frac{2}{n-2}\epsilon$, $s_\lambda$ is the global minimum point of $u_\lambda$, $M_{\lambda,-}=u_\lambda^-(s_\lambda)$ and $f$ is defined as in Proposition \ref{lbpp}. 
\end{prop}
\begin{rem}\label{uupbld}
The statement of the above proposition holds also for lower dimensions. More precisely, with small modification to the proof of Proposition \ref{lbnp} we have:
\begin{prop}\label{lbnpld}
Let $3 \leq n \leq 6$ and set
$$\tilde \lambda (n):=  \inf \{\lambda \in \R^+; \ \hbox{Problem (\ref{PBN2}) has a radial sign-changing solution}\ u_\lambda\}.$$
There exists $\bar\epsilon \in (0, \frac{n-2}{2})$ such that for all $\epsilon \in (0,\bar\epsilon)$ there exists $\delta=\delta(\epsilon) \in (0,1)$, with $\delta(\epsilon) \rightarrow 1$ as $\epsilon \rightarrow 0$, such that, for all $\lambda$ in a right neighborhood of $\tilde \lambda (n)$, (\ref{stimapp}) holds, where $M_{\lambda,-}=u_\lambda^-(s_\lambda)$, $s_\lambda$ is the global minimum point of $u_\lambda$ in the last nodal region \footnote{We assume without loss of generality that $u_\lambda$ is negative in that region.}.
\end{prop}
\end{rem}

\begin{proof}[Proof of Proposition \ref{lbnp}]
Let $v_\lambda$ the function defined by $v_\lambda(s):=u_\lambda^-(s+s_\lambda)$, $s \in (0,1-s_\lambda)$. Since $u_\lambda^-$ is a positive radial solution of (\ref{PBN}) then $v_\lambda$ is a solution of
 \begin{equation}\label{EF1}
\begin{cases}
v_\lambda^{\prime \prime} +\frac{n-1}{s+s_\lambda}v_\lambda^{\prime} + \lambda v_\lambda + v_\lambda^{2^* -1} =0  & \hbox{in}\ \ (0,1-s_\lambda)\\
v_\lambda^\prime( 0)=0, \ \ v_\lambda(1-s_\lambda)=0.&
\end{cases}
\end{equation}
 To eliminate $\lambda$ from the equation we make the following change of variable, $\rho:=\sqrt{\lambda} \ (s+s_\lambda)$, and we define $w_\lambda(\rho):=\lambda^{- \frac{n-2}{4}}v_\lambda(\frac{\rho}{\sqrt{\lambda}}-s_\lambda)=\lambda^{- \frac{n-2}{4}}u_\lambda^-(\frac{\rho}{\sqrt{\lambda}})$. By elementary computation we see that $w_\lambda$ solves
  \begin{equation}\label{EF2}
\begin{cases}
w_\lambda^{\prime \prime} +\frac{n-1}{\rho}w_\lambda^{\prime} + w_\lambda + w_\lambda^{2^* -1} =0  & \hbox{in}\ \ (\sqrt{\lambda}\  s_\lambda,\sqrt{\lambda})\\
w_\lambda^\prime( \sqrt{\lambda}\ s_\lambda)=0, \ \ w_\lambda(\sqrt{\lambda})=0.&
\end{cases}
\end{equation}
Making another change of variable, precisely $t:=\left(\frac{n-2}{\rho}\right)^{n-2}$, and setting
$y_\lambda(t):=w_\lambda\left(\frac{n-2}{t^{\frac{1}{n-2}}}\right)$ we eliminate the first derivative in (\ref{EF2}). Thus we get
  \begin{equation}\label{EF3}
\begin{cases}
y_\lambda^{\prime \prime} \ t^k + y_\lambda + y_\lambda^{2^* -1} =0  & \hbox{in}\ \ \left(\frac{(n-2)^{n-2}}{\lambda^{\frac{n-2}{2}}},\frac{(n-2)^{n-2}}{\lambda^{\frac{n-2}{2}}\ s_\lambda^{n-2}}\right),\\
y_\lambda^\prime\left( \frac{(n-2)^{n-2}}{\lambda^{\frac{n-2}{2}} s_\lambda^{n-2}}\right)=0, \ \ y_\lambda \left(\frac{(n-2)^{n-2}}{\lambda^{\frac{n-2}{2}}}\right)=0.&
\end{cases}
\end{equation}

where $\displaystyle k= 2 \frac{n-1}{n-2}>2$. To simplify the notation we set $t_{1,\lambda}:=\frac{(n-2)^{n-2}}{\lambda^{\frac{n-2}{2}}}$, $t_{2,\lambda}:=\frac{(n-2)^{n-2}}{\lambda^{\frac{n-2}{2}}\ s_\lambda^{n-2}}$, $I_\lambda=(t_{1,\lambda}, t_{2,\lambda})$ and $\gamma_\lambda:=y_\lambda(t_{2,\lambda})=\lambda^{- \frac{n-2}{4}}M_{\lambda,-}$. 
Observe also that $2^*-1=2k-3$.

We write the equation in (\ref{EF3}) as $y_\lambda^{\prime \prime} + t^{-k}(y_\lambda + y_\lambda^{2k-3}) =0$, which is an Emden-Fowler type equation $y^{\prime \prime} + t^{-k}h(y)=0$ with $h(y):=y+y^{2k-3}$.
The first step to prove (\ref{stimapp}) is the following inequality:

\begin{equation}\label{EF4}
(y_\lambda^\prime t^{k-1} y_\lambda^{1-k})^\prime + t^{k-2}y_\lambda^{-k}t_{2,\lambda}^{1-k}\gamma_\lambda h(\gamma_\lambda)\leq 0, \ \ \hbox{for all} \ t\in I_\lambda. 
\end{equation}
To prove (\ref{EF4}) we differentiate $y_\lambda^\prime t^{k-1} y_\lambda^{1-k}$. 
Since $y_\lambda^{\prime \prime} + t^{-k}h(y_\lambda)=0$ we get
\begin{eqnarray*}
&y_\lambda^{\prime\prime} t^{k-1} y_\lambda^{1-k} + y_\lambda^\prime (k-1)t^{k-2} y_\lambda^{1-k}-(k-1) (y_\lambda^\prime)^2 t^{k-1} y_\lambda^{-k}\\[8pt]
=&- t^{-k}(y_\lambda + y_\lambda^{2k-3}) t^{k-1} y_\lambda^{1-k} + y_\lambda^\prime (k-1)t^{k-2} y_\lambda^{1-k}-(k-1) (y_\lambda^\prime)^2 t^{k-1} y_\lambda^{-k}\\[8pt]
=&- t^{-1} y_\lambda^{2-k} - t^{-1} y_\lambda^{k-2} + y_\lambda^\prime (k-1)t^{k-2} y_\lambda^{1-k}-(k-1) (y_\lambda^\prime)^2 t^{k-1} y_\lambda^{-k}\\[8pt]
 =&-2(k-1)t^{k-2}y_\lambda^{-k}\left(\frac{1}{2(k-1)}t^{1-k}y_\lambda^2+\frac{1}{2(k-1)}t^{1-k}y_\lambda^{2k-2} - \frac{1}{2}y_\lambda y_\lambda^\prime + \frac{1}{2} t (y_\lambda^\prime)^2\right)\\[8pt]
  =&-2(k-1)t^{k-2}y_\lambda^{-k}\left(\frac{1}{2(k-1)}t^{1-k}y_\lambda h(y_\lambda) - \frac{1}{2}y_\lambda y_\lambda^\prime + \frac{1}{2} t (y_\lambda^\prime)^2\right).\\[8pt]
\end{eqnarray*}
Now we add and subtract the number $\frac{1}{2(k-1)}t_{2,\lambda}^{1-k}\gamma_\lambda h(\gamma_\lambda)$ inside the parenthesis, so we have
\begin{eqnarray*}
\begin{array}{ll}
&(y_\lambda^\prime t^{k-1} y_\lambda^{1-k})^\prime\\[8pt]
 =&-2(k-1)t^{k-2}y_\lambda^{-k}\left(\frac{1}{2(k-1)}t^{1-k}y_\lambda h(y_\lambda) - \frac{1}{2}y_\lambda y_\lambda^\prime + \frac{1}{2} t (y_\lambda^\prime)^2 -\frac{1}{2(k-1)}t_{2,\lambda}^{1-k}\gamma_\lambda h(\gamma_\lambda)\right)\\[8pt]
 &- t^{k-2}y_\lambda^{-k}t_{2,\lambda}^{1-k}\gamma_\lambda h(\gamma_\lambda).
 \end{array}
\end{eqnarray*}
Setting $L_\lambda(t):=\frac{1}{2(k-1)}t^{1-k}y_\lambda h(y_\lambda) - \frac{1}{2}y_\lambda y_\lambda^\prime + \frac{1}{2} t (y_\lambda^\prime)^2 -\frac{1}{2(k-1)}t_{2,\lambda}^{1-k}\gamma_\lambda h(\gamma_\lambda)$ we get

$$(y_\lambda^\prime t^{k-1} y_\lambda^{1-k})^\prime +  t^{k-2}y_\lambda^{-k}t_{2,\lambda}^{1-k}\gamma_\lambda h(\gamma_\lambda) = -2(k-1)t^{k-2}y_\lambda^{-k} L_\lambda(t).$$
If we show that $L_\lambda(t)\geq 0$ for all $t \in I_\lambda$ we get (\ref{EF4}). By definition it's immediate to verify that $L_\lambda(t_{2,\lambda})=0$, also by direct calculation we have $L_\lambda^\prime(t)=\frac{1}{2(k-1)} t^{1-k}y_\lambda^\prime [y_\lambda h^\prime(y_\lambda)-(2k-3)h(y_\lambda)]=\frac{1}{2(k-1)} t^{1-k}y_\lambda^\prime [(4-2k)y_\lambda] $. Since $y_\lambda>0$, $y_\lambda^\prime\geq0$  in $I_\lambda$ \footnote{$y_\lambda^\prime\geq 0$ because $(u_\lambda^-)^\prime(r) \leq 0$ for $s_\lambda<r<1$ as we can easily deduce from Corollary \ref{remutile}.} and $k>2$ we have $L_\lambda^\prime(t)\leq 0$ in $I_\lambda$, and from $L_\lambda(t_{2,\lambda})=0$ it follows $L_\lambda(t)\geq 0$ for all $t \in I_\lambda$.

As second step we integrate (\ref{EF4}) between $t$ and $t_{2,\lambda}$, for all $t \in I_\lambda$. Then, since $y_\lambda^\prime(t_{2,\lambda})=0$ we get 
$$-y_\lambda^\prime(t) t^{k-1} y_\lambda^{1-k}(t) + \int_{t}^{t_{2,\lambda}} s^{k-2}y_\lambda^{-k}(s)\ t_{2,\lambda}^{1-k}\gamma_\lambda h(\gamma_\lambda) \ ds \leq 0.$$

We rewrite this last inequality as
$$y_\lambda^\prime(t) t^{k-1} y_\lambda^{1-k}(t) \geq   t_{2,\lambda}^{1-k}\gamma_\lambda h(\gamma_\lambda) \ \int_{t}^{t_{2,\lambda}} s^{k-2}y_\lambda^{-k}(s)\ ds. $$
Since  $u_\lambda^- \leq M_{\lambda,-}$ by definition it follows $y_\lambda^{-k} \geq \gamma_\lambda^{-k}$,  so
\begin{eqnarray*}
\begin{array}{lll}
y_\lambda^\prime(t) t^{k-1} y_\lambda^{1-k}(t) & \geq&  \displaystyle t_{2,\lambda}^{1-k}\gamma_\lambda^{1-k} h(\gamma_\lambda)  \ \int_{t}^{t_{2,\lambda}} s^{k-2}\ ds\\[12pt]
&=&\displaystyle \frac{\gamma_\lambda^{1-k} h(\gamma_\lambda)}{k-1}  \ \frac{t_{2,\lambda}^{k-1}-t^{k-1}}{t_{2,\lambda}^{k-1}}\\[12pt]
&=&\displaystyle \frac{\gamma_\lambda^{1-k} h(\gamma_\lambda)}{k-1}  \left[1- \left(\frac{t}{t_{2,\lambda}}\right)^{k-1}\right].
\end{array}
\end{eqnarray*}
Multiplying the first and the last term of the above inequality by $t^{1-k}$ we get
$$\frac{1}{2-k}(y_\lambda^{2-k})^\prime(t)=y_\lambda^\prime(t)\  y_\lambda^{1-k}(t)  \geq \frac{\gamma_\lambda^{1-k} h(\gamma_\lambda)}{k-1}  \left(t^{1-k}- \frac{1}{t_{2,\lambda}^{k-1}}\right),$$
for all $t \in I_\lambda$. Integrating this inequality between $t$ and $t_{2,\lambda}$ we have

\begin{eqnarray*}
\begin{array}{lll}
\displaystyle \frac{\gamma_\lambda^{2-k}}{2-k} - \frac{y_\lambda^{2-k}(t)}{2-k}  &\geq& \displaystyle \frac{\gamma_\lambda^{1-k} h(\gamma_\lambda)}{k-1}  \int_{t}^{t_{2,\lambda}}\left(s^{1-k}- \frac{1}{t_{2,\lambda}^{k-1}}\right) ds \\[10pt]
&=& \displaystyle \frac{\gamma_\lambda^{1-k} h(\gamma_\lambda)}{k-1}  \left( \frac{t_{2,\lambda}^{2-k}}{2-k}-\frac{t^{2-k}}{2-k} - \frac{1}{t_{2,\lambda}^{k-2}}+ \frac{t}{{t_{2,\lambda}^{k-1}}}\right).
\end{array}
\end{eqnarray*}
We rewrite this last inequality as
\begin{equation}\label{EF5}
\begin{array}{lll}
\displaystyle \frac{y_\lambda^{2-k}(t)}{k-2}-\frac{\gamma_\lambda^{2-k}}{k-2} &\geq&\displaystyle  \frac{\gamma_\lambda^{1-k} h(\gamma_\lambda)}{k-1}  \left( \frac{t^{2-k}}{k-2}  + \frac{t}{{t_{2,\lambda}^{k-1}}} - \frac{k-1}{k-2} \frac{1}{t_{2,\lambda}^{k-2}}\right) \\[10pt]
&\geq&\displaystyle  \frac{\gamma_\lambda^{1-k} h(\gamma_\lambda)}{k-1} \ t^{2-k} \left[ \frac{1}{k-2}  + \left(\frac{t}{{t_{2,\lambda}}}\right)^{k-1} - \frac{k-1}{k-2} \left(\frac{t}{{t_{2,\lambda}}}\right)^{k-2}\right].
\end{array}
\end{equation}
To the aim of estimating the last term in (\ref{EF5}) we set $s:=\left(\frac{t}{{t_{2,\lambda}}}\right)^{k-1}$ and  study the function $g(s):=\frac{1}{k-2} + s -  \frac{k-1}{k-2} s^{\frac{k-2}{k-1}} $ in the interval $[0,1]$. Clearly  $g(0)=\frac{1}{k-2}=\frac{n-2}{2}>0$, $g(1)=0$ and $g$ is a decreasing function because $g^\prime(s)=1-s^{-\frac{1}{k-1}}< 0$ in $(0,1)$. In particular we have $g(s)>0$ in $(0,1)$. 
Let's fix $\epsilon \in (0, \frac{n-2}{2})$, by the monotonicity of $g$ we deduce that there exists only one $\delta=\delta(\epsilon) \in (0,1)$ such that $g(s)> \epsilon$ for all $0\leq s < \delta$, $g(\delta)=\epsilon$ and $\delta \rightarrow 1$ as $\epsilon \rightarrow 0$. Now remembering that $s=\left(\frac{t}{{t_{2,\lambda}}}\right)^{k-1}$, we have $\left(\frac{t}{{t_{2,\lambda}}}\right)^{k-1}<\delta$ if and only if $t < \delta^{\frac{1}{k-1}} t_{2,\lambda}$ and $t_{1,\lambda}<\delta^{\frac{1}{k-1}} t_{2,\lambda}$ if and only if $s_\lambda^{n-2} < \delta^{\frac{1}{k-1}}$ which is true for all $0<\lambda < \overline{\lambda}$, for some positive number $\overline{\lambda}=\overline{\lambda}(\epsilon)$. Setting $c(\epsilon):=(k-2) \epsilon$, from (\ref{EF5}) and the previous discussion we have
\begin{equation}\label{EF6}
\begin{array}{lll}
\displaystyle {y_\lambda^{2-k}(t)}-{\gamma_\lambda^{2-k}} \geq\displaystyle  \frac{\gamma_\lambda^{1-k} h(\gamma_\lambda)}{k-1} \ t^{2-k} c(\epsilon),
\end{array}
\end{equation}
for all $t \in (t_{1,\lambda},\  {\delta^{\frac{1}{k-1}}}t_{2,\lambda})$, $0<\lambda < \overline{\lambda}$.
Now from (\ref{EF6}) we deduce the desired bound for $u_\lambda^-$. In fact we have 
$$\displaystyle {y_\lambda^{2-k}(t)} \geq\displaystyle  {\gamma_\lambda^{2-k}}+ \frac{\gamma_\lambda^{1-k} h(\gamma_\lambda)}{k-1} \ t^{2-k} c(\epsilon),$$
from which, since $k>2$, we get

\begin{equation}\label{EF7}
\begin{array}{lll}
\displaystyle {y_\lambda(t)} &\leq&\displaystyle \left(  {\gamma_\lambda^{2-k}}+ \frac{\gamma_\lambda^{1-k} h(\gamma_\lambda)}{k-1} \ t^{2-k} c(\epsilon) \right)^{- \frac{1}{k-2}}\\[10pt]
&=&\displaystyle \gamma_\lambda \ \left(1+ \frac{\gamma_\lambda^{-1} h(\gamma_\lambda)}{k-1} \ t^{2-k} c(\epsilon) \right)^{- \frac{1}{k-2}}
\end{array}
\end{equation}
Now  by definition we have $y_\lambda(t)=\lambda^{- \frac{n-2}{4}} u_\lambda^-\left(\frac{\rho}{\sqrt{\lambda}}\right)=\lambda^{- \frac{n-2}{4}} u_\lambda^-(s+s_\lambda)$, $\gamma_\lambda=\lambda^{- \frac{n-2}{4}} M_{\lambda,-}$, $k-2=\frac{2}{n-2}$, $k-1=\frac{n}{n-2}$, $t=\left(\frac{n-2}{\rho}\right)^{n-2}=\left(\frac{n-2}{\sqrt{\lambda} (s+s_\lambda)}\right)^{n-2}$, in particular $t^{2-k}=t^{-\frac{2}{n-2}}=\left(\frac{\sqrt{\lambda} (s+s_\lambda)}{n-2}\right)^2=\frac{{\lambda} (s+s_\lambda)^2}{(n-2)^2}$. Thus we get 
\begin{equation*}
\begin{array}{lll}
\displaystyle \frac{\gamma_\lambda^{-1} h(\gamma_\lambda)}{k-1} \ t^{2-k} c(\epsilon)&=&\displaystyle \frac{\lambda^{\frac{n-2}{4}}M_{\lambda,-}^{-1}\left(\lambda^{- \frac{n-2}{4}} M_{\lambda,-} + \lambda^{- \frac{n+2}{4}}M_{\lambda,-}^{\frac{n+2}{n-2}}\right)}{\displaystyle \frac{n}{n-2}} c(\epsilon) \frac{\lambda (s+s_\lambda)^{2}}{(n-2)^2} \\[20pt]
&=&\displaystyle  \frac{M_{\lambda,-}^{-1}\left(\lambda M_{\lambda,-} + M_{\lambda,-}^{2^*-1}\right)}{n(n-2)} c(\epsilon) (s+s_\lambda)^{2}\\[10pt]
&=&\displaystyle \frac{M_{\lambda,-}^{-1}\ f\left(M_{\lambda,-}\right)}{n(n-2)} c(\epsilon) (s+s_\lambda)^{2},
\end{array}
\end{equation*}
where $f(z):=\lambda z + z^{2^*-1}$. Also by direct computation we see that the interval $(t_{1,\lambda},\  {\delta^{\frac{1}{k-1}}}t_{2,\lambda})$, corresponds to the interval $(\delta^{-\frac{1}{n}}s_\lambda, 1)$ for $s+s_\lambda=\frac{\rho}{\sqrt{\lambda}}=\frac{n-2}{\sqrt{\lambda}\ t^{\frac{1}{n-2}}}$. Thus from the previous computations and (\ref{EF7}) we have 
$$ \lambda^{- \frac{n-2}{4}} u_\lambda^-(s+s_\lambda) \leq {\lambda^{- \frac{n-2}{4}} M_{\lambda,-}}\ {\left(1+\frac{M_{\lambda,-}^{-1}\ f\left(\lambda M_{\lambda,-}\right)}{n(n-2)} c(\epsilon) (s+s_\lambda)^{2}\right)^{-\frac{n-2}{2}}}.$$
Finally dividing each term by $ \lambda^{- \frac{n-2}{4}}$ and setting $r:=s+s_\lambda$ we have
$$ u_\lambda^-(r) \leq  {\left(1+\frac{M_{\lambda,-}^{-1}\ f\left(\lambda M_{\lambda,-}\right)}{n(n-2)} c(\epsilon) r^{2}\right)^{-\frac{n-2}{2}}},$$
for all $r \in (\delta^{-\frac{1}{n}}s_\lambda, 1)$, which is the desired inequality since $u_\lambda^-$ is a radial function. 
\end{proof}

\section{Asymptotic analysis of the rescaled solutions}
\subsection{Rescaling the positive part}
As in Section 3 we consider a family $(u_\lambda)$ of least energy radial, sign-changing solutions of (\ref{PBN}) with $u_\lambda (0) >0$. Let us define $\beta:=\frac{2}{n-2}$, $\sigma_\lambda:=M_{\lambda, +}^\beta \cdot r_\lambda$; consider the rescaled function $\tilde u_\lambda^+(y)=\frac{1}{M_{\lambda, +}}u_\lambda^+\left(\frac{y}{M_{\lambda, +}^\beta}\right)$ in $B_{\sigma_\lambda}$. The following lemma is elementary but crucial.
\begin{lem}\label{ELLEM} We have:
\begin{description}
\item[(i)] $\|u_\lambda^+\|_{B_{r_\lambda}}^2=\|\tilde u_\lambda^+\|_{B_{\sigma_\lambda}}^2$,
\item[(ii)]$|u_\lambda^+|_{2^*,B_{r_\lambda}}^{2^*}=|\tilde u_\lambda^+|_{2^*,B_{\sigma_\lambda}}^{2^*}$,
\item[(iii)]$|u_\lambda^+|_{2,B_{r_\lambda}}^{2}=\frac{1}{M_{\lambda,+}^{2^*-2}}|\tilde u_\lambda^+|_{2,B_{\sigma_\lambda}}^{2}$
\end{description}
\end{lem}
\begin{proof}
To prove (i) we have only to remember the definition of $\tilde u_\lambda$ and make the change of variable $x \rightarrow \frac{y}{M_{\lambda, +}^\beta}$. Taking into account that by definition $\nabla_y \tilde u_\lambda^+(y)=\frac{1}{M_{\lambda, +}^{1+\beta}} (\nabla_x  u_\lambda^+)(\frac{y}{M_{\lambda, +}^\beta})$ and $2+2\beta=2+\frac{4}{n-2}= n\frac{2}{n-2}=n \beta = 2^*$, we get
\begin{equation*}
\begin{array}{lll}
\|u_\lambda^+\|_{B_{r_\lambda}}^2 &\displaystyle=\int_{B_{r_\lambda}} |\nabla_x u_\lambda^+(x)|^2 dx&\displaystyle= \frac{1}{M_{\lambda, +}^{n\beta}}\int_{B_{\sigma_\lambda}} \left|\nabla_x u_\lambda^+\left( \frac{y}{M_{\lambda, +}^\beta}\right)\right|^2 dy\\[12pt]
&\displaystyle =\frac{M_{\lambda, +}^{2+2\beta}}{M_{\lambda, +}^{n\beta}}\int_{B_{\sigma_\lambda}} \left|\nabla_y \tilde u_\lambda(y)\right|^2 dy&=\|\tilde u_\lambda^+\|_{B_{\sigma_\lambda}}^2.
\end{array}
\end{equation*}
The proof of (ii) is simpler:
\begin{equation*}
\begin{array}{lll}
\displaystyle \int_{B_{r_\lambda}} | u_\lambda^+(x)|^{2^*} dx&=&\displaystyle\int_{B_{\sigma_\lambda}} \frac{1}{M_{\lambda, +}^{n\beta}} \left| u_\lambda^+\left( \frac{y}{M_{\lambda, +}^\beta}\right)\right|^{2^*} dy\\[18pt]
&=&\displaystyle \int_{B_{\sigma_\lambda}} |\tilde u_\lambda^+(y) |^{2^*} dy.
\end{array}
\end{equation*}
The proof of (iii) is similar:
\begin{equation*}
\begin{array}{lll}
\displaystyle\int_{B_{r_\lambda}} | u_\lambda^+(x)|^{2} dx&=&\displaystyle\int_{B_{\sigma_\lambda}} \frac{1}{M_{\lambda, +}^{n\beta}} \left| u_\lambda^+\left( \frac{y}{M_{\lambda, +}^\beta}\right)\right|^{2} dy\\[18pt]
&=&\displaystyle \int_{B_{\sigma_\lambda}}\frac{1}{M_{\lambda, +}^{n\beta-2}} \left| \frac{1}{M_{\lambda, +}}u_\lambda^+\left( \frac{y}{M_{\lambda, +}^\beta}\right)\right|^{2} dy\\[18pt]
&=&\displaystyle\frac{1}{M_{\lambda, +}^{2^*-2}} \int_{B_{\sigma_\lambda}} |\tilde u_\lambda^+(y) |^{2} dy.
\end{array}
\end{equation*}
\end{proof}
\begin{rem}\label{santoro}
Obviously the previous lemma is still true if we consider any radial function $u \in H_{rad}^1(D)$, where $D$ is a radially symmetric domain in $\R^n$, and for any rescaling of the kind $\tilde u (y):=\frac{1}{M} u\left(\frac{y}{M^\beta}\right)$, where $M>0$ is a constant.
\end{rem}
The first qualitative result concerns the asymptotic behavior, as $\lambda \rightarrow 0$, of the radius $\sigma_\lambda=M_{\lambda, +}^\beta \cdot r_\lambda$ of the rescaled ball $B_{\sigma_\lambda}$. 
From Proposition \ref{LUS} we know that $r_\lambda \rightarrow 0$ as $\lambda \rightarrow 0$, so this result gives also information on the growth of $M_{\lambda, +}$ compared to the decay of $r_\lambda$.

\begin{prop} \label{rodota}
Up to a subsequence, $\sigma_\lambda \rightarrow + \infty$  as $\lambda \rightarrow 0$.
 \end{prop}
 \begin{proof}
  Up to a subsequence, as $\lambda \rightarrow 0$, we have three alternatives:
 \begin{description}
\item[(i)] $\sigma_\lambda \rightarrow 0$,
\item[(ii)] $\sigma_\lambda \rightarrow l >0$, $l\in\R$,
\item[(iii)] $\sigma_\lambda \rightarrow +\infty$.
\end{description}
We will show that (i) and (ii) cannot occur. Assume, by contradiction, that (i) holds then
writing $|u_\lambda^+|_{2^*, B_{r_\lambda}}^{2^*}$ in polar coordinates we have
\begin{eqnarray*}
\displaystyle |u_\lambda^+ |_{2^*, B_{r_\lambda}}^{2^*}&=&\displaystyle \omega_n \int_{0}^{r_\lambda}[u_\lambda^+(r)]^{2^*} r^{n-1} dr\\[8pt]
&\leq& \displaystyle \omega_n \ M_{\lambda, +}^{2^*}  \int_0^{r_\lambda} r^{n-1} dr \\[8pt]
&=&\displaystyle \omega_n \ (M_{\lambda,+}^{\beta})^n \ \frac{r_{\lambda}^{n}}{n}\\[8pt]
&=& \displaystyle \frac{\omega_n}{n}\ (M_{\lambda, +}^{\beta}\ r_{\lambda})^{n} \rightarrow 0 \ \ \ \hbox{as} \ \lambda\rightarrow 0.
\end{eqnarray*}
But from Proposition \ref{PROP1} we know that $ |u_\lambda^+ |_{2^*, B_{r_\lambda}}^{2^*} \rightarrow S^{n/2}$ as $\lambda \rightarrow 0$, so we get a contradiction.

Next assume by contradiction that (ii) holds. Since the rescaled functions $\tilde u_\lambda^+$ are solutions of
\begin{equation}\label{prematurata2}
\begin{cases}
-\Delta  u =\frac{\lambda}{M_\lambda^{2\beta}} u+   u^{2^* -1} & \hbox{in}\ B_{\sigma_\lambda}\\
 u > 0 & \hbox{in} \  B_{\sigma_\lambda}\\
 u =0 & \hbox{on} \ \partial B_{\sigma_\lambda}.
\end{cases}
\end{equation}
and $(\tilde u_\lambda^+)$ is uniformly bounded, then by standard elliptic theory, $\tilde u_\lambda^+ \rightarrow \tilde u $ in $C_{loc}^2(B_l)$, where $B_l$ is the limit domain of $B_{\sigma_\lambda}$ and $\tilde u$ solves
\begin{equation}\label{prematurata3}
\begin{cases}
-\Delta  u =   u^{2^* -1} & \hbox{in}\ B_{l}\\
 u > 0& \hbox{in}\ B_{l}.
\end{cases}
\end{equation}
Let us show that the boundary condition $\tilde u =0 \  \hbox{on} \ \partial B_{l}$ holds.
Since $M_{\lambda,+}$ is the global maximum of $u_\lambda$ (see Proposition \ref{TROLL}) then the rescaling $\tilde u_\lambda(y):=\frac{1}{M_{\lambda, +}}u_\lambda\left(\frac{y}{M_{\lambda, +}^\beta}\right)$ of the whole function $u_\lambda$  is a bounded solution of 

\begin{equation*}
\begin{cases}
-\Delta u =\frac{\lambda}{M_\lambda^{2\beta}}u+  |u|^{2^* -2} u & \hbox{in}\ B_{M_{\lambda,+}^\beta}\\
 u =0 & \hbox{on} \ \partial B_{M_{\lambda,+}^\beta}.
\end{cases}
\end{equation*}
So as before we get that $\tilde u_\lambda \rightarrow \tilde u_0$ in $C^2_{loc}(\R^n)$, where $\tilde u_0$ is a solution of $-\Delta u=  | u|^{2^* -2} u$ in $\R^n$. Obviously by definition we have  $\tilde u_\lambda (y)=\tilde u_\lambda^+(y)$ for all $y \in B_{\sigma_\lambda}$, $\tilde u_\lambda (y)=0$ for all $y \in \partial B_{\sigma_\lambda}$ and  $\tilde u_\lambda(y) <0 $ for all $y \in B_{M_{\lambda,+}^\beta} - \overline{B_{\sigma_\lambda}}$. Passing to the limit as $\lambda \rightarrow 0$,  since $\overline{B_l}$ is a compact set of $\R^n$ we have $\tilde u_\lambda \rightarrow \tilde u_0$ in $C^2(\overline{B_l})$, now  since $\tilde u=\tilde u_0> 0$ in $B_l$ and $\tilde u_0=0$ on $\partial B_l$, it follows $\tilde u=0$ on $\partial B_l$.
Since $B_l$ is a ball, by Pohozaev's identity, we know that the only possibility is $\tilde u \equiv 0$ which is a contradiction since $\tilde u(0)=1$. So the assertion is proved. 
 \end{proof}
 \begin{prop}\label{rescpp}
We have:
 \begin{equation}\label{ubpp}
 \tilde u_\lambda^+ (y)\leq \left\{1+  \frac{ 1}{n (n-2)}\left|{y}\right|^2\right\}^{-(n-2)/2},
\end{equation}
 for all $y \in \R^n$.
\end{prop}
 \begin{proof}
From (\ref{lbppeq}) for all $x \in B_{r_\lambda}$  we have 
 $$u_\lambda^+(x) \leq M_{\lambda,+} \left\{1+\frac{\lambda  + M_{\lambda,+}^{\frac{4}{n-2}}}{n (n-2)}|x|^2\right\}^{-(n-2)/2}.$$  
 Dividing each side by $M_{\lambda,+}$ and setting $ x=\frac{y}{M_{\lambda,+}^\beta}=\frac{y}{M_{\lambda,+}^{\frac{2}{n-2}}}$ we get
\begin{eqnarray*}
\begin{array}{lll}
\frac{1}{M_{\lambda,+}}u_\lambda^+\left(\frac{y}{M_{\lambda,+}^\beta}\right) &\leq& \left\{1+\frac{\lambda  + M_{\lambda,+}^{\frac{4}{n-2}}}{M_{\lambda,+}^{\frac{4}{n-2}} \  n (n-2)}\left|{y}\right|^2\right\}^{-(n-2)/2}\\[18pt]
&=&\left\{1+ \frac{\lambda}{M_{\lambda,+}^{\frac{4}{n-2}} \ n (n-2)}|y|^2 +  \frac{ 1}{n (n-2)}\left|{y}\right|^2\right\}^{-(n-2)/2}\\[18pt]
&\leq&\left\{1+  \frac{ 1}{n (n-2)}\left|{y}\right|^2\right\}^{-(n-2)/2},
 \end{array}
 \end{eqnarray*}
 for all $y \in B_{\sigma_\lambda}$. Thus we have proved (\ref{ubpp})  for all $y \in B_{\sigma_\lambda}$.
 Since $\tilde u_\lambda^+$ is zero outside the ball $B_{\sigma_\lambda}$ and the second term in (\ref{ubpp}) is independent of $\lambda$, this bound holds in the whole $\R^n$.
 \end{proof}

\subsection{An estimate on the first derivative at the node}
In this subsection we prove an inequality concerning $(u_\lambda^+)^\prime(r_\lambda)$ (or $(u_\lambda^-)^\prime(r_\lambda)$)  that will be useful in the next sections. 
\begin{lem}\label{LEM1}
There exists a constant $c_1$, depending only on $n$, such that 
\begin{equation}\label{stimdernodo}
|(u_\lambda^+)^\prime(r_\lambda) r_\lambda^{n-1}| \leq c_1\ r_\lambda^{\frac{n-2}{2}} 
\end{equation}
for all sufficiently small $\lambda>0$. Since $(u_\lambda^-)^\prime(r_\lambda) = -(u_\lambda^+)^\prime(r_\lambda)$ the same inequality holds for $(u_\lambda^-)^\prime(r_\lambda)$.
\end{lem}
\begin{proof}
Since $u_\lambda^+=u_\lambda^+(r)$ is a solution of $-[(u_\lambda^+)^\prime r^{n-1}]^\prime = \lambda u_\lambda^+ r^{n-1} + (u_\lambda^+)^{2^*-1} r^{n-1}$ in $(0,r_\lambda)$ and $(u_\lambda^+)^\prime(0)=0$ by integration we get 
\begin{eqnarray*}
(u_\lambda^+)^\prime(r_\lambda)  r_\lambda^{n-1} &=& - \left[  \int_0^{r_\lambda}\lambda u_\lambda^+ r^{n-1} dr +  \int_0^{r_\lambda}(u_\lambda^+)^{2^*-1} r^{n-1} dr\right]\\[8pt]
&=&\displaystyle -\left[  \frac{\lambda}{\omega_n}  \int_{B_{r_\lambda}}u_\lambda^+(x)\  dx +\frac{1}{\omega_n}  \int_{B_{r_\lambda}}[u_\lambda^+(x)]^{2^*-1} \ dx\right],
\end{eqnarray*}
where, as before,  $\omega_n$ denotes the measure of the $(n-1)$-dimensional unit sphere $S^{n-1}$. Using H\"{o}lder's inequality and observing that $meas(B_{r_\lambda})=\displaystyle \frac{\omega_n}{n} r_\lambda^n$ we deduce 
\begin{eqnarray*}
\left|(u_\lambda^+)^\prime(r_\lambda)  r_\lambda^{n-1} \right|&\leq \displaystyle   \frac{\lambda}{(n \ \omega_n)^{\frac{1}{2}}} r_\lambda^{\frac{n}{2}}|u_\lambda^+|_{2,B_{r_\lambda}} + \frac{1}{n^{\frac{n-2}{2n}}\ \omega_n^{\frac{n+2}{2n}}} r_\lambda^{\frac{n-2}{2}}\left[|u_\lambda^+|_{2^*,B_{r_\lambda}}^{2^*}\right]^{\frac{2^*-1}{2^*}}.
\end{eqnarray*}
From Proposition \ref{PROP1} we know that both $|u_\lambda^+|_{2,B_{r_\lambda}}$, $|u_\lambda^+|_{2^*,B_{r_\lambda}}^{2^*}$ are bounded, moreover from Proposition \ref{LUS} we have $r_\lambda \rightarrow 0$ as $\lambda \rightarrow 0$. So there exists a constant $c_1=c_1(n)$ such that for all sufficiently small $\lambda>0$ (\ref{stimdernodo}) holds.
\end{proof}
\subsection{Rescaling the negative part}

Now we study the rescaled function $\tilde u_\lambda^- (y):=\frac{1}{M_{\lambda, -}}u_\lambda^-\left(\frac{y}{M_{\lambda, -}^\beta}\right)$ in the annulus $A_{\rho_\lambda}:=\{y \in \R^n; M_{\lambda, -}^\beta r_\lambda <|y|<M_{\lambda, -}^\beta\}$, where $\rho_\lambda:=M_{\lambda, -}^\beta r_\lambda$. This case is more delicate than the previous one since the radius $s_\lambda$, where the the minimum is achieved, depends on $\lambda$.
Thus, roughly speaking, we have to understand how $r_\lambda$ and  $s_\lambda$ behave with respect to the scaling parameter $M_{\lambda,-}^\beta$. This means that we  have to study the asymptotic behavior of $M_{\lambda,-}^\beta r_\lambda$ and $M_{\lambda,-}^\beta s_\lambda$ as $\lambda \rightarrow 0$. It will be convenient to consider also the one-dimensional rescaling  
$$ z_\lambda(s):=\frac{1}{M_{\lambda, -}}u_\lambda^-\left(s_\lambda+\frac{s}{M_{\lambda, -}^\beta}\right),$$
which satisfies
\begin{equation}\label{ORPBN}
\begin{cases}
z_\lambda^{\prime \prime} +\frac{n-1}{s+M_{\lambda, -}^\beta s_\lambda}z_\lambda^{\prime} + \frac{\lambda}{M_{\lambda, -}^{2\beta}} z_\lambda + z_\lambda^{2^* -1} =0  & \hbox{in}\ \ (a_\lambda,b_\lambda)\\
z_\lambda^\prime( 0)=0, \ \ z_\lambda(0)=1,&
\end{cases}
\end{equation}
where $a_\lambda:=M_{\lambda, -}^\beta \cdot (r_\lambda - s_\lambda)<0$,  $b_\lambda:=M_{\lambda, -}^\beta \cdot (1 - s_\lambda)>0$.
We define $\gamma_\lambda := M_{\lambda, -}^\beta s_\lambda$.

Since $s_\lambda \rightarrow 0$ as $\lambda \rightarrow 0$, we have $b_\lambda \rightarrow +\infty$; for the remaining parameters $a_\lambda, \gamma_\lambda$ it will suffice to study the asymptotic behavior of $\gamma_\lambda$ as $\lambda \rightarrow 0$.

Up to a subsequence we have three alternatives:
\begin{description}
\item[(a)] $\gamma_\lambda \rightarrow +\infty$,
\item[(b)] $\gamma_\lambda \rightarrow \gamma_0> 0$,
\item[(c)] $\gamma_\lambda \rightarrow 0$.
\end{description}
\begin{lem} \label{lemmacassazione}
$\gamma_\lambda \rightarrow +\infty$ cannot happen.
\end{lem}
\begin{proof}
Assume  $\gamma_\lambda \rightarrow +\infty$; 
up to a subsequence we have $a_\lambda \rightarrow \bar a\leq 0$, as $\lambda \rightarrow 0$, where $\bar a \in \R \cup \{-\infty\}$.

If $\bar a<0$ or $\bar a = - \infty$ then passing to the limit in (\ref{ORPBN}) as $\gamma_\lambda=M_{\lambda, -}^\beta \cdot s_\lambda \rightarrow +\infty$ we have that $z_\lambda \rightarrow z$ in $C_{loc}^1(\bar a,+\infty)$, where $z$ solves the limit problem
\begin{equation}\label{OLRPBN}
\begin{cases}
z^{\prime \prime} + z^{2^* -1} =0  & \hbox{in}\ \ (\bar a,+\infty)\\
z^\prime( 0)=0, \ \ z(0)=1.&
\end{cases}
\end{equation}
Since $z_\lambda \rightarrow z$ in $C_{loc}^1(\bar a,+\infty)$ and being $z_\lambda >0$, then  by Fatou's lemma we have
$$\liminf_{\lambda \rightarrow 0}\int_{a_\lambda}^{b_\lambda} [z_\lambda(s)]^{2^*} ds \geq  \int_{\bar a}^{+\infty} [z(s)]^{2^*} ds \geq c_1>0.$$
In particular, being $a_\lambda<0$, by the same argument it follows that for all small $\lambda>0$ $$\int_{0}^{b_\lambda} [z_\lambda(s)]^{2^*} ds \geq \int_{0}^{+\infty} [z(s)]^{2^*} ds \geq c_2 > 0.$$

Now we have the following estimate: 
\begin{eqnarray*}
\begin{array}{lllll}
\displaystyle |u_\lambda^- |_{2^*, A_{r_\lambda}}^{2^*}&=&\displaystyle \omega_n \int_{r_\lambda}^{1}[u_\lambda^-(r)]^{2^*} r^{n-1} dr
&\geq& \displaystyle \omega_n s_{\lambda}^{n-1} \int_{s_\lambda}^{1}[u_\lambda^-(r)]^{2^*} dr \\[10pt]
&=&\displaystyle \omega_n s_{\lambda}^{n-1} M_{\lambda,-}^{2^*} \int_{s_\lambda}^{1}\left[\frac{1}{M_{\lambda,-}}u_\lambda^-(r)\right]^{2^*} dr &=& \displaystyle \omega_n s_{\lambda}^{n-1} M_{\lambda,-}^{2^*-\beta} \int_{0}^{b_\lambda}[z_\lambda(s)]^{2^*}  ds\\[10pt]
&= &\displaystyle \omega_n \gamma_\lambda^{n-1}  \int_{0}^{b_\lambda}[z_\lambda(s)]^{2^*}  ds& \geq&\displaystyle \omega_n \gamma_\lambda^{n-1}  c_2,\\[10pt]
\end{array}
\end{eqnarray*}
having used the change of variable  $r=s_\lambda + \frac{s}{M_{\lambda,-}^\beta}$. Since $|u_\lambda^- |_{2^*, A_{r_\lambda}}^{2^*} \rightarrow S^{n/2}$ while $\gamma_\lambda \rightarrow + \infty$, as $\lambda \rightarrow 0$, we get a contradiction. 

If instead $\bar a=0$ we consider the rescaled function  $\tilde u_\lambda^-$ which solves 
\begin{equation}\label{TPBN}
\begin{cases}
-\Delta \tilde u_\lambda = \frac{\lambda}{M_{\lambda, -}^{2\beta}} \tilde u_\lambda + \tilde u_\lambda^{2^* -1} & \hbox{in}\ A_{\rho_\lambda}\\
\tilde u =0 & \hbox{on}\ \partial A_{\rho_\lambda},
\end{cases}
\end{equation}
  and is uniformly bounded. 
We observe that since $a_\lambda \rightarrow 0$ then $\rho_\lambda=a_\lambda + \gamma_\lambda \rightarrow + \infty$.
By definition we have $\tilde u_\lambda^-(\rho_\lambda) = 0$, $\tilde u_\lambda^-(\gamma_\lambda)=1 $,  for all $\lambda \in (0,\lambda_1)$.
Thus we have
 $$\frac{|\tilde u_\lambda^-(\rho_\lambda) -\tilde u_\lambda^-(\gamma_\lambda)|}{|\rho_\lambda - \gamma_\lambda|}=\frac{1}{|a_\lambda|}\rightarrow +\infty \ \ \hbox{as} \ \lambda\rightarrow 0. $$ 
 
From standard elliptic regularity theory we know that  $\tilde u_\lambda^-$ is a classical solution, so by the mean value theorem, $$\frac{|\tilde u_\lambda^-(\rho_\lambda) -\tilde u_\lambda^-(\gamma_\lambda)|}{|\rho_\lambda - \gamma_\lambda|}=|(\tilde u_\lambda^-)^\prime(\xi_\lambda)|,$$ for some $\xi_\lambda \in (\rho_\lambda, \gamma_\lambda)$; thus $|(\tilde u_\lambda^-)^\prime(\xi_\lambda)| \rightarrow +\infty$ as $\lambda \rightarrow 0$. From Corollary \ref{remutile} it follows that 
 $(\tilde u_\lambda^-)^\prime >0$ in $(\rho_\lambda,\gamma_\lambda)$ for all $\lambda>0$. 

 By writing (\ref{TPBN}) in polar coordinates we get: $$(\tilde u_\lambda^-)^{\prime \prime} +\frac{n-1}{r}(\tilde u_\lambda^-)^{\prime} + \frac{\lambda}{M_{\lambda,-}^{2\beta}} \tilde u_\lambda^- + (\tilde u_\lambda^- )^{2^* -1} =0 .$$
  From this, since $\tilde u_\lambda^- > 0$ and $(\tilde u_\lambda^-)^\prime >  0$ in $(\rho_\lambda,\gamma_\lambda)$, we get $(\tilde u_\lambda^-)^{\prime\prime} < 0$ in $(\rho_\lambda,\gamma_\lambda)$. Thus $(\tilde u_\lambda^-)^\prime(\rho_\lambda) > (\tilde u_\lambda^-)^\prime(\xi_\lambda)>0 $, for all $\lambda >0$. In particular $(\tilde u_\lambda^-)^\prime(\rho_\lambda) \rightarrow + \infty$ as $\lambda \rightarrow 0$.

Since, by elementary computation, we have $(\tilde u_\lambda^-)^\prime(\rho_\lambda)=\frac{1}{M_{\lambda,-}^{1+\beta}}(u_\lambda^-)^\prime(r_\lambda)$, by Lemma  \ref{LEM1} we get
 $$|(\tilde u_\lambda^-)^\prime(\rho_\lambda)| \leq c \frac{1}{M_{\lambda,-}^{1+\beta} \ r_\lambda^{n/2}} $$
for a constant c independent from $\lambda$. Remembering that $1+\beta= 1+\frac{2}{n-2}=\beta \cdot \frac{n}{2}$, and the definition of $\rho_\lambda$ we have the following estimate
$$|(\tilde u_\lambda^-)^\prime(\rho_\lambda)| \leq c \frac{1}{\rho_\lambda^{n/2}}.$$
Since  $\rho_\lambda \rightarrow + \infty$, as $\lambda \rightarrow 0$, we deduce that $(\tilde u_\lambda^-)^\prime(\rho_\lambda)$ is uniformly bounded, against $(\tilde u_\lambda^-)^\prime(\rho_\lambda) \rightarrow + \infty$ as $\lambda \rightarrow 0$. Thus we get a contradiction. 
\end{proof}
Thanks to Lemma \ref{lemmacassazione} we deduce that $(\gamma_\lambda)$ is a bounded sequence. The following proposition states an uniform upper bound for $\tilde u_\lambda^-$.

 \begin{prop}\label{rescpn}
 Let's fix $\epsilon \in (0,\frac{n-2}{2})$, and set $\bar M:=\sup_\lambda \gamma_\lambda$. There exist $h=h(\epsilon)$ and $\bar \lambda=\bar \lambda(\epsilon)>0$ such that
\begin{equation}\label{Geulbnp}
\tilde u_\lambda^-(y) \leq U_{h}(y)
\end{equation}
for all $y \in \R^n$, $0<\lambda<\bar\lambda$, where
\begin{equation}\label{uupperbound} 
U_{h}(y):=\begin{cases} 
1 & \hbox{if } \ |y|\leq h \\
 \left[1+\frac{1}{n (n-2)} c(\epsilon)|y|^2\right]^{-(n-2)/2} &\hbox{if }  \ |y|>h,
\end{cases}
\end{equation}
with $c(\epsilon)= \frac{2}{n-2} \epsilon$.
\end{prop}

\begin{proof}
 We fix $\epsilon\in (0,\frac{n-2}{2})$, so by Proposition \ref{lbnp} there exist $\delta=\delta(\epsilon) \in (0,1)$ and $\overline{\lambda}(\epsilon)>0$ such that 
 $$u_\lambda^-(x) \leq M_{\lambda,-} \left\{1+\frac{M_{\lambda,-}^{-1}\ f(M_{\lambda,-})}{n (n-2)} c(\epsilon)|x|^2\right\}^{-(n-2)/2}, $$
for all $x \in A_{\delta, \lambda}= \{ x \in \R^n; \ \delta^{- 1/N}s_\lambda <|x|<1\}$, for all $\lambda \in (0,\overline{\lambda})$, where $c(\epsilon)=\frac{2}{n-2}\epsilon$.
 The same proof of Proposition \ref{rescpp} shows that
 $$\tilde u_\lambda^-(y) \leq   \left\{1+\frac{1}{n (n-2)} c(\epsilon)|y|^2\right\}^{-(n-2)/2},$$
for all $y \in \tilde{A}_{\delta, \lambda}=\{ y \in \R^n; \ M_{\lambda,-}^\beta \delta^{- 1/N}s_\lambda <|y|<M_{\lambda,-}^\beta\}$. Now since by definition $\tilde u_\lambda^-$ is uniformly bounded by 1 we get an upper bound defined in the whole annulus $\tilde{A}_{\rho_\lambda}= \{ y \in \R^n; \ M_{\lambda,-}^\beta r_\lambda <|y|<M_{\lambda,-}^\beta\}$; to be more precise $\tilde u_\lambda^-(y) \leq U_{\lambda}(y)$, where
\begin{equation}\label{provv}
U_{\lambda}(y):=\begin{cases} 
1 & \hbox{if } \ M_{\lambda,-}^\beta r_\lambda <|y|\leq M_{\lambda,-}^\beta \delta^{- 1/N}s_\lambda \\
 \left[1+\frac{1}{n (n-2)} c(\epsilon)|y|^2\right]^{-(n-2)/2} &\hbox{if }  \ M_{\lambda,-}^\beta \delta^{- 1/N}s_\lambda <|y|<M_{\lambda,-}^\beta.
\end{cases}
\end{equation}
Since  $\gamma_\lambda =M_{\lambda,-}^\beta s_\lambda \leq \bar M$, then setting $h:=\delta^{- 1/N}\bar M$ we get that $ \delta^{- 1/N}M_{\lambda,-}^\beta s_\lambda \leq h$. Therefore,  from  (\ref{provv}), since $\tilde u_\lambda^-$ is zero outside $\tilde{A}_{\rho_\lambda}$, we deduce (\ref{Geulbnp}). 
 \end{proof}

\begin{lem}\label{minneo}
$\gamma_\lambda \rightarrow \gamma_0 >0$, $\gamma_0 \in\R$, cannot happen.
\end{lem}
\begin{proof}
Assume that $\gamma_\lambda \rightarrow \gamma_0>0$, $\gamma_0\in \R$. Since $0< r_\lambda < s_\lambda$ there are only two possibilities for $a_\lambda$. To be precise, up to a subsequence we can have:
\begin{description}
\item[(i)] $a_\lambda \rightarrow 0$,
\item[(ii)] $a_\lambda \rightarrow \bar a < 0$, $\bar a\in\R$.
\end{description}
We will show that both (i) and (ii) lead to a contradiction.

If we assume (i) the same proof of Lemma \ref{lemmacassazione} gives a contradiction. We point out that now $\rho_\lambda \rightarrow \gamma_0$, as $\lambda \rightarrow 0$, so as before we get a contradiction since $(\tilde u_\lambda^-)^\prime(\rho_\lambda)$ is uniformly bounded, against $(\tilde u_\lambda^-)^\prime(\rho_\lambda) \rightarrow + \infty$ as $\lambda \rightarrow 0$.

Assuming (ii) we have $a_\lambda \rightarrow \bar a<0$ and $\gamma_\lambda \rightarrow \gamma_0>0$. We define $m:=\bar a+\gamma_0$. Clearly we have $0\leq m<\gamma_0$ and $\rho_\lambda \rightarrow m$ as $\lambda \rightarrow 0$. 
Assume $m>0$ and consider the rescaling $\tilde u_\lambda^-$ in the annulus $A_{\rho_\lambda}$ defined as before. Since $\tilde u_\lambda^-$ satisfies (\ref{TPBN}) and $(\tilde u_\lambda^-)$ is uniformly bounded
then passing to the limit as $\lambda \rightarrow 0$ we get $\tilde u_\lambda^- \rightarrow  \tilde u$ in $C_{loc}^2(\Pi)$, where $\Pi$ is the limit domain $\Pi:=\{y \in \R^n;|y|>m \}$ and $\tilde u$ is a positive radial solution of
 \begin{eqnarray}\label{TLPBN}
-\Delta  \tilde u =   \tilde u^{2^* -1} & \hbox{in}\ \Pi
\end{eqnarray}
By definition $\tilde u_\lambda^-(\gamma_\lambda)=1$, $(\tilde u_\lambda^-)^\prime(\gamma_\lambda)=0$ for all $\lambda$, so as $\lambda \rightarrow 0$ we get $\tilde u(\gamma_0)=1$, $\tilde u^\prime(\gamma_0)=0$ because of the convergence of $\tilde u_\lambda^- \rightarrow \tilde u$ in $C^2(K)$, for all compact subsets $K$ in $\Pi$, and $\gamma_0>m$. In particular we deduce that $\tilde u\not\equiv 0$.
We now show that $\tilde u$ can be extended to zero on $\partial \Pi=\{y \in\R^n; |y|=m\}$. Thanks to Lemma \ref{LEM1} and since we are assuming $m>0$, which is the limit of $\rho_\lambda$ as $\lambda \rightarrow 0$, we get that $(\tilde u_\lambda^-)^\prime (\rho_\lambda)$ is uniformly bounded by a constant $M$, and by the monotonicity of $(\tilde u_\lambda^-)^\prime$ the same bound holds for $(\tilde u_\lambda^-)^\prime (s)$ for all $s \in (\rho_\lambda,\gamma_\lambda)$. It follows that in that interval $\tilde u_\lambda^-(s) \leq M(s-\rho_\lambda)$. Passing to the limit as $\lambda \rightarrow 0$ we have $\tilde u (s) \leq M(s-m)$ for all $s \in (m,\gamma_0)$ which implies $\tilde u$ can be extended by continuity to zero on $\partial \Pi$. We use the same notation $\tilde u$ to denote this extension.

Observe that $\tilde u$ has finite energy, in particular, 
using Fatou's lemma and thanks to Lemma \ref{ELLEM}, Remark \ref{santoro}, Proposition \ref{PROP1}, we get
  \begin{equation}\label{gradesttu}
 \int_{\Pi}|\nabla  \tilde u|^2 dy \leq  \liminf_{\lambda \rightarrow 0} \int_{A_{\rho_\lambda}} |\nabla\tilde u_\lambda^-|^2 dy=\liminf_{\lambda \rightarrow 0} \int_{A_{r_\lambda}} |\nabla u_\lambda^-|^2 dx =S^{n/2},
 \end{equation}
 \begin{equation}\label{gradesttu2}
  \int_{\Pi}|  \tilde u|^{2*} dy \leq  \liminf_{\lambda \rightarrow 0} \int_{A_{\rho_\lambda}}|\tilde u_\lambda^-|^{2*}dy= \liminf_{\lambda \rightarrow 0} \int_{A_{r_\lambda}}| u_\lambda^-|^{2*}dx=S^{n/2}.
\end{equation}
Moreover, since $\tilde u_\lambda^- \rightarrow \tilde u$ in $C_{loc}^2(\Pi)$ and thanks to the uniform upper bound given by Proposition \ref{rescpn}, by Lebesgue's theorem we have 
 \begin{equation}\label{energialimite}
  \int_{\Pi}|  \tilde u|^{2*} dy = \lim_{\lambda \rightarrow 0} \int_{A_{r_\lambda}}| u_\lambda^-|^{2*}dx=S^{n/2}.
\end{equation}
Since $\tilde u \in H^1(\Pi)\cap C^0(\bar\Pi)$ and is zero on $\partial\Pi$, then $\tilde u \in H_0^{1}(\Pi)$ and thanks to (\ref{gradesttu}), (\ref{energialimite}) it follows that $\tilde u$ achieves the best constant in the Sobolev embedding on $\Pi$, which is impossible (see for instance \cite{STRUWE}, Theorem III.1.2). This ends the proof for the case $m>0$.

Assume now $m=0$, then $\tilde u_\lambda^-$ converges in $C^2_{loc}(\R^n-\{0\})$ to a radial function $\tilde u$ which is a positive bounded solution of 
 \begin{eqnarray}\label{NONSING}
-\Delta  \tilde u =   \tilde u^{2^* -1} & \hbox{in}\ \ \R^n-\{0\}
\end{eqnarray}

Since $\tilde u$ is a radial solution of (\ref{NONSING}), then integrating $-(\tilde u^\prime(r) r^{n-1})^\prime = \tilde u^{2^*-1}(r)r^{n-1} $ between $\delta>0$ sufficiently small and $\gamma_0$ we get 
$$\tilde u^\prime(\delta) \delta ^{n-1} = \int_{\delta}^{\gamma_0}\tilde u^{2^*-1}r^{n-1}dr. $$
Since the right hand side  is a positive and decreasing function of $\delta$, we get $\tilde u^\prime(\delta) \delta ^{n-1} \rightarrow \tilde l>0$ as $\delta \rightarrow 0$. Thus $\tilde u^\prime(\delta)$ behaves as $\delta^{1-n}$ near the origin and this is a contradiction since $\int_{\R^n}|  \nabla\tilde u|^{2} dy= \omega_n \int_{0}^{+ \infty} |\tilde u^\prime(r)|^2 r^{n-1} dr$ is finite, and the proof is complete.
\end{proof}
As a consequence of Lemma \ref{lemmacassazione} and Lemma \ref{minneo} we have proved:

\begin{prop}\label{asgamma}
Up to a subsequence we have $\gamma_\lambda \rightarrow 0$  
as $\lambda \rightarrow 0$.
\end{prop}

\subsection{Final estimates and proof of Theorem \ref{mainteo}}
From Proposition \ref{asgamma} we know that, up to a subsequence, $\gamma_\lambda=M_{-,\lambda}^\beta s_\lambda \rightarrow 0$  as $\lambda \rightarrow 0$. 
The rescaled function $\tilde u_\lambda^- (y):=\frac{1}{M_{\lambda, -}}u_\lambda^-\left(\frac{y}{M_{\lambda, -}^\beta}\right)$  in the annulus $A_{\rho_\lambda}:=\{y \in \R^n; M_{\lambda, -}^\beta r_\lambda <|y|<M_{\lambda, -}^\beta\}$ solves (\ref{TPBN}) and the functions $(\tilde u_\lambda^-)$ are uniformly bounded. 
Since $\gamma_\lambda \rightarrow 0$  as $\lambda \rightarrow 0$, in particular 
the limit domain of $A_{\rho_\lambda}$ is $\R^n-\{0\}$ and  by standard elliptic theory
$\tilde u_\lambda^- \rightarrow \tilde u$ in $C^2_{loc}(\R^n-\{0\})$, where $\tilde u$ is positive, radial and solves 
\begin{eqnarray}\label{WSLPBNf}
-\Delta  \tilde u =   \tilde u^{2^* -1} & \hbox{in}\ \R^n-\{0\}
\end{eqnarray} 
As in the proof of Lemma \ref{minneo} by Fatou's Lemma it follows that $\tilde u$ has finite energy $I_0(\tilde u)=\frac{1}{2}|\nabla\tilde u|_{2,\R^n}^2-\frac{1}{2^*}|\tilde u|_{2^*,\R^n}^{2^*}$. 
Moreover, thanks to the uniform upper bound (\ref{Geulbnp}), by Lebesgue's theorem we have 
$$\lim_{\lambda \rightarrow 0} \int_{A_{\rho_\lambda}}|\tilde u_\lambda^-|^{2*}dy =\int_{\R^n}|  \tilde u|^{2*} dy,$$
so, by Lemma \ref{ELLEM}, Remark \ref{santoro} and Proposition \ref{PROP1} we get  
$$\int_{\R^n}| \tilde u|^{2*} dy = S^{n/2}.$$
The next two lemmas show that the function $\tilde u=\tilde u(s)$ can be extended to a $C^1([0,+\infty))$ function if we set $\tilde u(0):=1$ and $\tilde u^\prime(0):=0$.

\begin{lem}\label{lemdel1}
We have $$\lim_{s\rightarrow 0} \tilde u (s)=1.$$
\end{lem}
\begin{proof}
Since $\tilde u_\lambda^-$ is a radial solution of (\ref{TPBN}) and $\tilde u_\lambda^-\leq 1$, then 
\begin{equation*}
\begin{array}{lll}
[(\tilde u_\lambda^-)^\prime s^{n-1}]^\prime &=&-\frac{\lambda}{M_{\lambda,-}^{2\beta}}\tilde u_\lambda^-(s) s^{n-1}-[\tilde u_\lambda^-(s)]^{2^*-1}s^{n-1}\\[8pt]
&\geq&-\frac{\lambda}{M_{\lambda,-}^{2\beta}} s^{n-1}-s^{n-1}\\[8pt]
&\geq&-2s^{n-1}.
\end{array}
\end{equation*}
Integrating between $\gamma_\lambda$ and $s>\gamma_\lambda$ (with $s<M_{\lambda,-}^{\beta}$) we get 
\begin{equation*}
(\tilde u_\lambda^-)^\prime(s) s^{n-1} \geq\displaystyle -2\int_{\gamma_\lambda}^{s}t^{n-1} dt
\geq-\frac{2}{n} s^{n}.
\end{equation*}
Hence $(\tilde u_\lambda^-)^\prime(s) \geq -\frac{2}{n} s$ for all  $s \in (\gamma_\lambda, M_{\lambda,-}^{\beta})$. Integrating again between $\gamma_\lambda$ and $s$ we have
\begin{equation*}
\tilde u_\lambda^-(s)-1 \geq -\frac{1}{n} (s^{2}-\gamma_\lambda^2)\geq-\frac{1}{n} s^{2}.
\end{equation*}
Hence $\tilde u_\lambda^-(s) \geq 1-\frac{1}{n} s^{n}$ for all $s \in (\gamma_\lambda, M_{\lambda,-}^{\beta})$. Since $\gamma_\lambda \rightarrow 0$ and $M_{\lambda,-}^{\beta} \rightarrow +\infty$, then, passing to the limit as $\lambda \rightarrow 0$, we get  $\tilde u(s) \geq 1-\frac{1}{n} s^{2}$, for all $s>0$. From this inequality and since $\tilde u \leq 1$ we deduce $\lim_{s\rightarrow 0} \tilde u(s)=1$.
\end{proof}

\begin{lem}\label{lemdel2}
We have 
$$\lim_{s \rightarrow 0} \tilde u^\prime (s)=0.$$
\end{lem}
\begin{proof}
As before, from the radial equation satisfied by $\tilde u_\lambda^-$, integrating between $\gamma_\lambda$ and $s>\gamma_\lambda$ (with $s<M_{\lambda,-}^{\beta}$) we get 
\begin{equation*}
\begin{array}{lll}
-(\tilde u_\lambda^-)^\prime(s) s^{n-1} &=&\displaystyle \frac{\lambda}{M_{\lambda,-}^{2\beta}} \int_{\gamma_\lambda}^{s}\tilde u_\lambda^-t^{n-1} dt + \int_{\gamma_\lambda}^{s} (\tilde u_\lambda^-)^{2^*-1}t^{n-1} dt .
\end{array}
\end{equation*}
Since $\tilde u \leq 1$, and $\gamma_\lambda \rightarrow 0$ it follows that for all $\lambda>0$ sufficiently small 
\begin{equation*}
|(\tilde u_\lambda^-)^\prime(s) s^{n-1}| \leq \displaystyle \frac{\lambda}{M_{\lambda,-}^{2\beta}} \int_{\gamma_\lambda}^{s} t^{n-1} dt + \int_{\gamma_\lambda}^{s}t^{n-1} dt \leq \displaystyle 2 \frac{s^n}{n}.
\end{equation*}
Passing to the limit, as $\lambda \rightarrow 0$, we get  $\displaystyle |\tilde u^\prime(s)|  \leq 2 \frac{s}{n}$ for all $s>0$, hence $\lim_{s\rightarrow 0}\tilde u^\prime(s)=0$.
\end{proof}

From Lemma \ref{lemdel1} and Lemma \ref{lemdel2} it follows that the radial function $\tilde u(y)=\tilde u(|y|)$ can be extended to a $C^1(\R^n)$ function. From now on we denote by $\tilde u$ this extension. Next lemma shows that $\tilde u$ is a weak solution of (\ref{WSLPBNf}) in the whole $\R^n$.
\begin{lem}\label{fedistrsol}
The function $\tilde u$ is a weak solution of 
\begin{eqnarray}\label{festproblim}
-\Delta  \tilde u =   \tilde u^{2^* -1} & \hbox{in}\ \R^n
\end{eqnarray}
\end{lem}
\begin{proof}
Let's fix a test function $\phi \in C_0^\infty (\R^n)$. If $0 \notin supp(\phi)$ the proof is trivial so from now on we assume $0 \in  supp(\phi)$. Let $B(\delta)$ be the ball  centered at the origin having radius $\delta>0$, with $\delta$ sufficiently small such that $supp (\phi)\subset\subset B(1/\delta)$. 
Applying Green's formula to $\Omega(\delta):=B(1/\delta)-B(\delta)$, since $\tilde u$ is a $C_{loc}^2(\R^n-\{0\})$ solution of (\ref{WSLPBNf}) and  $\phi\equiv 0$ on $\partial B(1/\delta)$, we have

\begin{equation}\label{eq1fest}
 \int_{\Omega(\delta)} \nabla \tilde u \cdot \nabla \phi \ dy =\int_{\Omega(\delta)} \phi\  \tilde u^{2^*-1} \ dy + \int_{\partial B(\delta)} \phi \left(\frac{\partial \tilde u}{\partial \nu}\right)\ d\sigma .
 \end{equation}
  We show now that $\int_{\partial B(\delta)} \phi \left(\frac{\partial \tilde u}{\partial \nu}\right)\ d\sigma \rightarrow 0$ as $\delta \rightarrow 0$. In fact since $\tilde u$ is a radial function we have $\frac{\partial \tilde u}{\partial \nu}(y)= \tilde u^\prime (\delta) $ for all $y \in \partial B(\delta)$, and from this relation we get
\begin{eqnarray*}
\left| \int_{\partial B(\delta)} \phi \left(\frac{\partial \tilde u}{\partial \nu}\right)\ d\sigma\right|
&\leq&| \tilde u^\prime (\delta)| \int_{\partial B(\delta)} |\phi| \ d\sigma\\
&\leq& \omega_n| \tilde u^\prime (\delta)| \delta^{n-1} ||\phi||_{\infty}.
\end{eqnarray*}
Thanks to Lemma \ref{lemdel2} we have $| \tilde u^\prime (\delta)| \delta^{n-1}\rightarrow 0$ as $\delta \rightarrow 0$.
To complete the proof we pass to the limit in (\ref{eq1fest}) as $\delta \rightarrow 0$. We observe that 
\begin{equation}\label{eq1wsol}
\begin{array}{lll}
\displaystyle |\nabla \tilde u \cdot \nabla \phi| \ \chi_{\Omega(\delta)} 
 &\leq&\displaystyle | \nabla \tilde u|^2 \ \chi_{\{| \nabla \tilde u|>1\}}|\nabla \phi| + | \nabla \tilde u| \ \chi_{\{| \nabla \tilde u|\leq1\}}|\nabla \phi|\\[10pt]
 &\leq&\displaystyle  | \nabla \tilde u|^2 \ \chi_{\{| \nabla \tilde u|>1\}}|\nabla \phi| +\  \chi_{\{| \nabla \tilde u|\leq1\}}|\nabla \phi|.
\end{array}
\end{equation}
Since $\int_{\R^n}|  \nabla\tilde u|^{2} dy \leq  S^{n/2}$ and $\phi$ has compact support the right-hand side of (\ref{eq1wsol}) belongs to $L^1(\R^n)$.
Hence from Lebesgue's theorem we have
\begin{equation}\label{eq2fest}
 \lim_{\delta \rightarrow 0} \int_{\Omega(\delta)} \nabla \tilde u \cdot \nabla \phi \ dy = \int_{\R^n} \nabla \tilde u \cdot \nabla \phi \ dy.
\end{equation}
Since $\phi$ has compact support by Lebesgue's theorem we have 
\begin{equation}\label{eq3fest}
\lim_{\delta \rightarrow 0}\int_{\Omega(\delta)} \phi\  \tilde u^{2^*-1} \ dy = \int_{\R^n} \phi\  \tilde u^{2^*-1} \ dy.
\end{equation}
From  (\ref{eq1fest}), (\ref{eq2fest}), (\ref{eq3fest}) and since we have proved $\int_{\partial B(\delta)} \phi \left(\frac{\partial \tilde u}{\partial \nu}\right)\ d\sigma \rightarrow 0$ as $\delta \rightarrow 0$ it follows that
\begin{equation*}
 \int_{\R^n} \nabla \tilde u \cdot \nabla \phi \ dy =\int_{\R^n} \phi\  \tilde u^{2^*-1} \ dy,
 \end{equation*}
 which completes the proof.
\end{proof}

Now we have all the tools to prove Theorem \ref{mainteo}.

\begin{proof}[Proof of Theorem \ref{mainteo}]
We start proving (i). By Proposition \ref{rodota}, arguing as in the previous proofs we know that $(\tilde u_\lambda^+)$ is an equi-bounded family of radial solutions  of (\ref{prematurata2}) and converges in $C^2_{loc}(\R^n)$ to a function $\tilde u$ which solves $-\Delta  u =   u^{2^* -1}$ in $\R^n$. From (\ref{ubpp}) we deduce that $\tilde u \rightarrow 0$ as  $|y| \rightarrow + \infty$. To apply Proposition \ref{UNLP} we have to check that $\tilde u$ has finite energy, but this is an immediate consequence of Fatou's lemma  and the assumption that $u_\lambda$ has finite energy (for the details see (\ref{gradesttu}) and (\ref{gradesttu2})). Thus $\tilde u=\delta_{x_0,\mu}$ for some $x_0 \in \R^n$, $\mu>0$. Since $\tilde u$ is a radial function we have $x_0=0$. Moreover, since $\tilde u(0)=1$, by an elementary computation we see that $\mu=\sqrt{n(n-2)}$. 

Now we prove (ii). As we have seen at the beginning of this section the equi-bounded family $(\tilde u_\lambda^-)$ converges in $C_{loc}^2(\R^n-\{0\})$ to a function $\tilde u$ which solves (\ref{WSLPBNf}). From Lemma \ref{lemdel1} and Lemma \ref{lemdel2} we have that $\tilde u$ can be extended to a $C^1(\R^n)$ function such that $\tilde u(0)=1$, $\nabla \tilde u(0)=0$.
Moreover from Lemma \ref{fedistrsol} we know that $\tilde u $ is a weak solution of (\ref{festproblim}) and from Fatou's lemma, as seen in (\ref{gradesttu}), (\ref{gradesttu2}), we have that $\tilde u$ has finite energy. Also from Proposition \ref{rescpn} we deduce that  $\tilde u \rightarrow 0$ as  $|y| \rightarrow + \infty$. 

By elliptic regularity (see for instance Appendix B of \cite{STRUWE}) since $\tilde u$ is a weak solution of (\ref{festproblim}) we deduce that $\tilde u \in C^2(\R^n)$. Thanks to Proposition \ref{UNLP}, since $\tilde u$ is a radial function and $\tilde u(0)=1$, we have $\tilde u= \delta_{0,\mu}$, where $\mu>0$ is the same as in (i).
\end{proof}

\section{Asymptotic behavior of $M_{\lambda,+}$, $M_{\lambda,-}$ and proof of Theorem \ref{asestpropmain}}
We know from Proposition \ref{PROP1} that $M_{\lambda,+}, M_{\lambda,-} \rightarrow +\infty$ as $\lambda \rightarrow 0$, in addition in the last two sections we have proved that $M_{\lambda,+}^\beta r_\lambda \rightarrow + \infty $ while $M_{\lambda,-}^\beta r_\lambda \rightarrow 0 $, as $\lambda \rightarrow 0$. Thus $\frac{M_{\lambda,+}}{M_{\lambda,-}}\rightarrow +\infty $ as $\lambda \rightarrow 0$; in other words $M_{\lambda,+}$ goes to infinity faster than $M_{\lambda,-}$. In this section we determine the order of infinity of $M_{\lambda,-}$ as negative power of $\lambda$ and also an asymptotic relation between $M_{\lambda,+}$, $M_{\lambda,-}$ and the node $r_\lambda$. 

\begin{prop}\label{asestprop1}
As $\lambda\rightarrow 0$ we have
\begin{description}
\item[(i)] $M_{\lambda,+}|(u_\lambda^+)^\prime (r_\lambda)| r_\lambda^{n-1}\rightarrow c_1(n)$;
\item[(ii)] $\lambda^{-1}M_{\lambda,+}^{2\beta} r_\lambda^{n} |(u_\lambda^+)^\prime (r_\lambda)|^2\rightarrow c_2(n)$;
\item[(iii)] $M_{\lambda,+}^{2-2\beta} r_\lambda^{n-2}\lambda \rightarrow c_3(n)$,
\end{description}
where $c_1(n)=\int_{0}^{\infty} \delta_{0,\mu}^{2^*-1}(s) s^{n-1}ds$, $c_2(n)=2\int_{0}^{\infty} \delta_{0,\mu}^{2}(s) s^{n-1}ds$, $c_3(n)=\frac{c_1^2(n)}{c_2(n)}$.
\end{prop}
\begin{proof}
To prove (i) we integrate the equation $-[(u_\lambda^+)^\prime r^{n-1}]^\prime= \lambda u_\lambda^+ r^{n-1} + (u_\lambda^+)^{2^*-1}r^{n-1}$ between $0$ and $r_\lambda$ and multiply both sides by $M_{\lambda,+}$. Since $(u_\lambda^+)^\prime(0)=0$ we have
\begin{equation}\label{diesis}
M_{\lambda,+} |(u_\lambda^+)^\prime(r_\lambda)| r_\lambda^{n-1}= \lambda M_{\lambda,+} \int_0^{r_\lambda} u_\lambda^+ r^{n-1} \ dr+   M_{\lambda,+} \int_0^{r_\lambda}(u_\lambda^+)^{2^*-1}r^{n-1} \ dr.
\end{equation}
We first prove that $\lambda M_{\lambda,+} \int_0^{r_\lambda} u_\lambda^+ r^{n-1} \ dr \rightarrow 0$ as $\lambda \rightarrow 0$. In fact by the usual change of variable $r=\frac{s}{M_{\lambda,+}^\beta}$ we have
\begin{equation*}
\begin{array}{lll}
\lambda \displaystyle M_{\lambda,+} \int_0^{r_\lambda} u_\lambda^+(r) \ r^{n-1} \ dr &=& \displaystyle \lambda \frac{1}{M_{\lambda,+}^{2^*-2}} \int_0^{M_{\lambda,+}^\beta r_\lambda} \frac{1}{M_{\lambda,+}} u_\lambda^+\left(\frac{s}{M_{\lambda,+}^\beta}\right) s^{n-1} \ ds\\[18pt]
&=& \displaystyle \lambda \frac{1}{M_{\lambda,+}^{2^*-2}} \int_0^{M_{\lambda,+}^\beta r_\lambda} \tilde u_\lambda^+(s) s^{n-1} \ ds
\end{array}
\end{equation*}
Thanks to the uniform upper bound (\ref{ubpp}) we have
\begin{equation*}
\begin{array}{lll}
 \displaystyle \lambda \frac{1}{M_{\lambda,+}^{2^*-2}} \int_0^{M_{\lambda,+}^\beta r_\lambda} \tilde u_\lambda^+ \ s^{n-1} \ ds &\leq& \displaystyle  \lambda \frac{1}{M_{\lambda,+}^{2^*-2}} \int_0^{M_{\lambda,+}^\beta r_\lambda}  \left\{1+  \frac{ 1}{n (n-2)}{s}^2\right\}^{-(n-2)/2} s^{n-1} ds\\[10pt]
 &\leq& \displaystyle  \lambda \frac{1}{M_{\lambda,+}^{2^*-2}} \int_0^{1} s^{n-1} ds \\[10pt]
 &+&\displaystyle \lambda \frac{1}{M_{\lambda,+}^{2^*-2}}[n(n-2)]^{(n-2)/2} \int_1^{M_{\lambda,+}^\beta r_\lambda}  {s}^{-(n-2)} s^{n-1} ds\\[12pt]
 &=& I_{\lambda,1} + I_{\lambda,2}.
\end{array}
\end{equation*}
Since $M_{\lambda,+} \rightarrow +\infty$ and $\int_0^{1} s^{n-1} ds = \frac{1}{n}$ it's obvious that  $I_{\lambda,1} \rightarrow 0$, as $\lambda \rightarrow 0$. Now we show that the same holds for $ I_{\lambda,2}$. In fact, setting $C_1(n):=[n(n-2)]^{(n-2)/2}$ we have
\begin{equation*}
\begin{array}{lll}
\displaystyle I_{\lambda,2}&=&\displaystyle \lambda \frac{1}{M_{\lambda,+}^{2^*-2}}C_1(n) \int_1^{M_{\lambda,+}^\beta r_\lambda}  {s} \ ds\\[12pt]
&=&\displaystyle \lambda \frac{1}{M_{\lambda,+}^{2^*-2}}C_1(n) \left(\frac{M_{\lambda,+}^{2\beta} r_\lambda^2}{2} -\frac{1}{2}\right)\\[12pt]
&=&\displaystyle \lambda  r_\lambda^2\frac{C_1(n)}{2} -\lambda \frac{1}{M_{\lambda,+}^{2^*-2}}\frac{C_1(n)}{2}\rightarrow 0, \ \hbox{as } \lambda \rightarrow 0,
\end{array}
\end{equation*}
since by definition, $2\beta=\frac{4}{n-2}=2^*-2$.
To complete the proof of (i) we show that $M_{\lambda,+} \int_0^{r_\lambda}(u_\lambda^+)^{2^*-1}r^{n-1} \ dr \rightarrow \int_{0}^{\infty} \delta_{0,\mu}^{2^*-1}(s) s^{n-1}ds$ as $\lambda \rightarrow 0$. In fact, as before, by the change of variable $r=\frac{s}{M_{\lambda,+}^\beta}$ we have
\begin{equation*}
\begin{array}{lll}
 \displaystyle M_{\lambda,+} \int_0^{r_\lambda}[u_\lambda^+(r)]^{2^*-1} \ r^{n-1} \ dr &=& \displaystyle \frac{1}{M_{\lambda,+}^{2^*-1}} \int_0^{M_{\lambda,+}^{\beta} r_\lambda}\left[u_\lambda^+\left(\frac{s}{M_{\lambda,+}^\beta}\right)\right]^{2^*-1} s^{n-1} \ ds\\[18pt]
 &=& \displaystyle \int_0^{M_{\lambda,+}^{\beta} r_\lambda}[\tilde u_\lambda^+(s)]^{2^*-1} s^{n-1} \ ds.
\end{array}
\end{equation*}
Since $\tilde u_\lambda^+ \rightarrow \delta_{0,\mu} $ in $C_{loc}^2(\R^n)$, in particular we have $[\tilde u_\lambda^+(s)]^{2^*-1} \rightarrow [\delta_{0,\mu}(s)]^{2^*-1}$ as $\lambda \rightarrow 0$,  for all $s\geq 0$, and thanks to the uniform upper bound (\ref{ubpp}), by Lebesgue's dominated convergence theorem, it follows that  $ \int_0^{M_{\lambda,+}^{\beta} r_\lambda}[\tilde u_\lambda^+(s)]^{2^*-1} s^{n-1} \ ds \rightarrow \int_{0}^{\infty} \delta_{0,\mu}^{2^*-1}(s) s^{n-1}ds$ so by (\ref{diesis}) the proof of (i) is complete.

Now we prove (ii). Applying Pohozaev's identity to $u_\lambda^+$, which solves $-\Delta u=\lambda u + u ^{2^*-1}$ in $B_{r_\lambda}$, we have
$$\lambda \int_{B_{r_\lambda}} [u_\lambda^+(x)]^2 \ dx = \frac{1}{2} \int_{\partial B_{r_\lambda}}(x\cdot\nu) \left(\frac{\partial u_\lambda^+}{\partial \nu}\right)^2 \ d\sigma,$$
where $\nu$ is the exterior unit normal vector to $\partial B_{r_\lambda}$.  Since $u_\lambda^+$ is radial we have also $ \left(\frac{\partial u_\lambda^+}{\partial \nu}\right)^2= \left[(u_\lambda^+)^\prime(r_\lambda) \right]^2$ so, passing to the unit sphere $S^{n-1}$, we get
\begin{eqnarray*}
\lambda \int_{B_{r_\lambda}} [u_\lambda^+(x)]^2 \ dx &=& \frac{1}{2} r_\lambda^{n-1}\int_{S^{n-1}}r_\lambda \left[(u_\lambda^+)^\prime(r_\lambda) \right]^2  \ d\omega\\
&=& \frac{1}{2} \omega_n r_\lambda^{n} \left[(u_\lambda^+)^\prime(r_\lambda) \right]^2.
\end{eqnarray*}
Thus we have 
\begin{equation}\label{asest1}
\lambda^{-1} r_\lambda^{n} \left[(u_\lambda^+)^\prime(r_\lambda) \right]^2= 2\ \omega_n^{-1} \int_{B_{r_\lambda}} [u_\lambda^+(x)]^2 \ dx.
 \end{equation}
Now performing the same change of variable as in (i) we have 
\begin{eqnarray*}
\int_{B_{r_\lambda}} [u_\lambda^+(x)]^2 \ dx&=&\frac{1}{M_{\lambda,+}^{2^*-2}} \int_{B_{\sigma_\lambda}} \left[\frac{1}{M_{\lambda,+}}u_\lambda^+\left(\frac{y}{M_{\lambda,+}^\beta}\right)\right]^2 \ dy\\
&=&\frac{1}{M_{\lambda,+}^{2^*-2}} \int_{B_{\sigma_\lambda}} \left[\tilde u_\lambda^+\left(y \right)\right]^2 \ dy,
\end{eqnarray*}

Thus we get 
\begin{equation}\label{asest2}
M_{\lambda,+}^{2\beta} \int_{B_{r_\lambda}} [u_\lambda^+(x)]^2 \ dx =\int_{B_{\sigma_\lambda}} \left[\tilde u_\lambda^+\left(y \right)\right]^2 \ dy.
 \end{equation}

As in (i) since $\tilde u_\lambda^+ \rightarrow \delta_{0,\mu} $ in $C_{loc}^2(\R^n)$ and thanks to the uniform upper bound (\ref{ubpp}) we have $$\int_{B_{\sigma_\lambda}} \left[\tilde u_\lambda^+\left(y \right)\right]^2 \ dy \rightarrow \int_{\R^n} [\delta_{0,\mu}(y)]^2 \ dy=\omega_n \int_{0}^{+\infty} [\delta_{0,\mu}(r)]^2 r^{n-1}\ dr.$$
From this, (\ref{asest1}) and (\ref{asest2}) we deduce that 
$\lambda^{-1} M_{\lambda,+}^{2\beta} r_\lambda^{n} \left[(u_\lambda^+)^\prime(r_\lambda) \right]^2\rightarrow 2 \int_{0}^{+\infty} [\delta_{0,\mu}(r)]^2 r^{n-1}\ dr$, and (ii) is proved.

The proof of (iii) is a trivial consequence of (i) and (ii).
\end{proof}

Now we state a similar result for $M_{\lambda,-}$.

\begin{prop}\label{asestprop2}
As $\lambda\rightarrow 0$ we have the following:
\begin{description}
\item[(i)] $M_{\lambda,-}|(u_\lambda^-)^\prime (1)| \rightarrow c_1(n)$;
\item[(ii)] $\lambda^{-1}M_{\lambda,-}^{2\beta}  \left\{[(u_\lambda^-)^\prime (1)]^2 -[(u_\lambda^-)^\prime (r_\lambda)]^2 r_\lambda^{n}\right\}\rightarrow c_2(n)$;
\item[(iii)] $\lambda^{-1}M_{\lambda,-}^{2\beta}  [(u_\lambda^-)^\prime (r_\lambda)]^2 r_\lambda^{n}\rightarrow 0$;
\item[(iv)] $M_{\lambda,-}^{2-2\beta} \lambda \rightarrow c_3(n)$,
\end{description}
where $c_1(n)$, $c_2(n)$ and $c_3(n)$ are the constants defined in Proposition \ref{asestprop1}.
\end{prop}
\begin{proof}
The proof of (i) is similar to the proof of (i) of Proposition \ref{asestprop1}. Here we integrate the equation $-[(u_\lambda^-)^\prime r^{n-1}]^\prime= \lambda u_\lambda^- r^{n-1} + (u_\lambda^-)^{2^*-1}r^{n-1}$ between $s_\lambda$ and $1$. Since $ (u_\lambda^-)^\prime(s_\lambda)=0$  we have
$$(u_\lambda^-)^\prime(1) = \lambda \int_{s_\lambda}^{1} u_\lambda^-\  r^{n-1} \ dr+   \int_{s_\lambda}^{1}(u_\lambda^-)^{2^*-1}r^{n-1} \ dr.$$
By $M_\lambda^\beta s_\lambda\rightarrow 0$ and   thanks to the uniform upper bound (\ref{Geulbnp}), arguing like in the proof of (i) of Proposition \ref{asestprop1},
we have $$M_{\lambda,-} \ \lambda \int_{s_\lambda}^{1} u_\lambda^-\  r^{n-1} \ dr \rightarrow 0$$ 
and 
$$ M_{\lambda,-} \int_{s_\lambda}^{1}(u_\lambda^-)^{2^*-1}r^{n-1} \ dr = \int_{M_{\lambda,-} ^\beta s_\lambda}^{M_{\lambda,-} ^\beta} (\tilde u_\lambda^-)^{2^*-1}s^{n-1} \ ds \rightarrow  \int_{0}^{+ \infty} \delta_{0,\mu}^{2^*-1}s^{n-1} \ ds,$$
as $\lambda \rightarrow 0$. The proof of (i) is complete.

The proof of (ii) is similar to the corresponding one of Proposition \ref{asestprop1}. This time we apply Pohozaev's identity to $u_\lambda^-$ in the annulus $A_{r_\lambda}= \{x \in \R^n; \ r_\lambda<|x|<1\}$ whose boundary has two connected components, namely $\{x \in \R^n; \ |x|= r_\lambda\}$ and the unit sphere $S^{n-1}$. Thus we have
\begin{eqnarray*}
\lambda \int_{A_{r_\lambda}} [u_\lambda^-(x)]^2 dx &=& \frac{1}{2}  \int_{\partial A_{r_\lambda}} (x\cdot\nu)\left(\frac{\partial u_\lambda^-}{\partial \nu}\right)^2 \ d\sigma\\
&=& \frac{1}{2} \omega_n \left\{ [(u_\lambda^-)^\prime(1)]^2 - [(u_\lambda^-)^\prime(r_\lambda)]^2 r_\lambda^{n}\right\}.
\end{eqnarray*}
Thus multiplying each member by $M_{\lambda,-}^{2\beta}$ and rewriting the previous equation we have
\begin{eqnarray*}
M_{\lambda,-}^{2\beta} \lambda^{-1} \left\{[(u_\lambda^-)^\prime(1)]^2- [(u_\lambda^-)^\prime(r_\lambda)]^2 r_\lambda^{n}\right\} &=& 2\omega_n^{-1} M_{\lambda,-}^{2\beta} \int_{A_{r_\lambda}} [u_\lambda^-(x)]^2 dx\\
&=& 2\omega_n^{-1} M_{\lambda,-}^{2\beta}\frac{1}{M_{\lambda,-}^{n\beta}} \int_{A_{\sigma_\lambda}} \left[u_\lambda^-\left(\frac{y}{M_{\lambda,-}^{\beta}}\right)\right]^2 dy\\
&=& 2  \int_{M_{\lambda,-}^{\beta}r_\lambda}^{M_{\lambda,-}^{\beta}} \left[\tilde u_\lambda^-(s)\right]^2 \ s^{n-1}ds.
\end{eqnarray*}
Since $2  \int_{M_{\lambda,-}^{\beta}r_\lambda}^{M_{\lambda,-}^{\beta}} \left[\tilde u_\lambda^-(s)\right]^2 \ s^{n-1}ds \rightarrow 2\int_{0}^{\infty} \delta_{0,\mu}^{2}(s) s^{n-1}ds$ as $\lambda \rightarrow 0$ we are done. 

To prove (iii) we write 
\begin{eqnarray*}
\lambda^{-1}M_{\lambda,-}^{2\beta}  [(u_\lambda^-)^\prime (r_\lambda)]^2 r_\lambda^{n}&=&\frac {\lambda^{-1}M_{\lambda,-}^{2\beta}  [(u_\lambda^-)^\prime (r_\lambda)]^2 r_\lambda^{n}}{\lambda^{-1}M_{\lambda,+}^{2\beta}  [(u_\lambda^+)^\prime (r_\lambda)]^2 r_\lambda^{n}} \cdot \lambda^{-1}M_{\lambda,+}^{2\beta}  [(u_\lambda^+)^\prime (r_\lambda)]^2 r_\lambda^{n}\\[12pt]
&=&\frac {M_{\lambda,-}^{2\beta} }{M_{\lambda,+}^{2\beta} } \cdot \ \lambda^{-1}M_{\lambda,+}^{2\beta}  [(u_\lambda^+)^\prime (r_\lambda)]^2 r_\lambda^{n}\rightarrow 0
\end{eqnarray*}
 since $\frac {M_{\lambda,-}}{M_{\lambda,+}}  \rightarrow 0$ and  $\lambda^{-1}M_{\lambda,+}^{2\beta}  [(u_\lambda^+)^\prime (r_\lambda)]^2 r_\lambda^{n}\rightarrow c_2(n)$ as $\lambda \rightarrow 0$ (by (ii) of Proposition  \ref{asestprop1}).
 
Finally the proof of (iv) is trivial. In fact from (ii) and (iii) it immediately follows that $$\lambda^{-1}M_{\lambda,-}^{2\beta}  [(u_\lambda^-)^\prime (1)]^2\rightarrow c_2(n).$$
 From this and (i), we get (iv).
\end{proof}

From (iii) of Proposition  \ref{asestprop1} and (iv) of Proposition  \ref{asestprop2} we deduce  the following result which gives an asymptotic relation between $M_{\lambda,+}$, $M_{\lambda,-}$ and $r_\lambda$.
\begin{prop}
$\displaystyle \frac{M_{\lambda,-}^{2-2\beta}}{M_{\lambda,+}^{2-2\beta}r_\lambda^{n-2}} \rightarrow 1$, {as} $\ \lambda \rightarrow 0$.
\end{prop}

\begin{rem}\label{ovvio}
By elementary computation $2- 2\beta=2-\frac{4}{n-2}=\frac{2n-8}{n-2}$ so by (iv) of Proposition  \ref{asestprop2} we have that $M_{\lambda,-}$ is an infinite of the same order as $\lambda^{-\frac{n-2}{2n-8}}$.
\end{rem}

\section{Proof of Theorem  \ref{greentea}}
This section is entirely devoted to the proof of Theorem  \ref{greentea}.

\begin{proof}[Proof of Theorem  \ref{greentea}]
We want to prove that $\lambda^{-\frac{n-2}{2n-8}} u_\lambda \rightarrow \tilde c(n) G(x,0)$ in $C_{loc}^1(B_1-\{0\})$. We begin from the local uniform convergence of $\lambda^{-\frac{n-2}{2n-8}} u_\lambda$. The same argument with some modifications will work for the local uniform convergence of its derivatives. 
Thanks to the representation formula, since $-\Delta u_\lambda = \lambda u_\lambda + |u_\lambda|^{2^*-2}u_\lambda$ in $B_1$, we have
\begin{equation}\label{formrappres}
\lambda^{-\frac{n-2}{2n-8}} u_\lambda(x) =- \lambda^{-\frac{n-2}{2n-8}} \lambda \int_{B_1} G(x,y) u_\lambda(y) \ dy -  \lambda^{-\frac{n-2}{2n-8}} \int_{B_1} G(x,y)  |u_\lambda|^{2^*-2}u_\lambda(y) \ dy.
\end{equation}
Since  $ \lambda^{-\frac{n-2}{2n-8}} \lambda = \lambda^{\frac{n-6}{2n-8}}$, splitting the integrals we have
\begin{eqnarray*}
\begin{array}{lllll}
\displaystyle \lambda^{-\frac{n-2}{2n-8}} u_\lambda(x) &=&- \displaystyle  \lambda^{\frac{n-6}{2n-8}} \int_{B_{r_\lambda}} G(x,y) u_\lambda^+(y) \ dy  &+&  \displaystyle \lambda^{\frac{n-6}{2n-8}} \int_{A_{r_\lambda}} G(x,y) u_\lambda^-(y) \ dy\\[12pt]
&&\displaystyle -  \lambda^{-\frac{n-2}{2n-8}} \int_{B_{r_\lambda}} G(x,y)  [u_\lambda^+(y)]^{2^*-1} \ dy &+& \displaystyle   \lambda^{-\frac{n-2}{2n-8}} \int_{A_{r_\lambda}} G(x,y)  [u_\lambda^-(y)]^{2^*-1} \ dy\\[12pt]
&=&I_{1,\lambda} + I_{2,\lambda} +I_{3,\lambda} +I_{4,\lambda}.&&
\end{array}
\end{eqnarray*}

Let $K$ be a compact subset of $B_1-\{0\}$. We are going to prove that
$I_{1,\lambda}, I_{2,\lambda}, I_{3,\lambda} \rightarrow 0$ uniformly in $K$, as $\lambda\rightarrow 0$.  We begin with $I_{1,\lambda}$. For all $x \in K$ we have 

\begin{eqnarray*}
\begin{array}{lll}
\displaystyle |I_{1,\lambda}|&\leq& \displaystyle \left| \lambda^{\frac{n-6}{2n-8}} \int_{B_{r_\lambda}} G(x,y) u_\lambda^+(y) \ dy\right|\\[16pt] 
&=&\displaystyle \left| \lambda^{\frac{n-6}{2n-8}} \frac{1}{M_{\lambda,+}^{n\beta}}\int_{B_{M_{\lambda,+}^{\beta}r_\lambda}} G\left(x,\frac{y}{M_{\lambda,+}^{\beta}}\right) u_\lambda^+\left(\frac{y}{M_{\lambda,+}^{\beta}}\right) \ dy\right|\\[16pt]
&\leq&\displaystyle  \lambda^{\frac{n-6}{2n-8}} \frac{1}{M_{\lambda,+}^{2^*-1}}\int_{B_{M_{\lambda,+}^{\beta}r_\lambda}} \left|G\left(x,\frac{y}{M_{\lambda,+}^{\beta}}\right)\right| \tilde u_\lambda^+\left(y\right) \ dy.
\end{array}
\end{eqnarray*}
Since $K$ is a compact subset of $B_1-\{0\}$ and $|\frac{y}{M_{\lambda,+}^{\beta}}|< r_\lambda$ by an elementary computation we see that for all $x \in K$, for all $\lambda>0$ sufficiently small $\left|G\left(x,\frac{y}{M_{\lambda,+}^{\beta}}\right)\right| \leq c (K)$ for all $y \in B_{M_{\lambda,+}^{\beta}r_\lambda} $, where $c=c(K)$ is a positive costant depending only on $K$ and $n$. Now thanks to the uniform upper bound (\ref{ubpp}) we have
\begin{eqnarray*}
\begin{array}{lll}
&&\displaystyle  \lambda^{\frac{n-6}{2n-8}} \frac{1}{M_{\lambda,+}^{2^*-1}}\int_{B_{M_{\lambda,+}^{\beta}r_\lambda}} \left|G\left(x,\frac{y}{M_{\lambda,+}^{\beta}}\right)\right| \tilde u_\lambda^+\left(y\right) \ dy\\[16pt]
&\leq&\displaystyle c(K) \lambda^{\frac{n-6}{2n-8}} \frac{1}{M_{\lambda,+}^{2^*-1}}\int_{B_{M_{\lambda,+}^{\beta}r_\lambda}} \left\{1+  \frac{ 1}{n (n-2)}\left|{y}\right|^2\right\}^{-(n-2)/2}
\ dy\\[16pt]
&=&\displaystyle c(K) \lambda^{\frac{n-6}{2n-8}} \frac{1}{M_{\lambda,+}^{2^*-1}} \omega_n \int_0^{{M_{\lambda,+}^{\beta}r_\lambda}} \left\{1+  \frac{ 1}{n (n-2)}s^2\right\}^{-(n-2)/2}s^{n-1}\
\ ds\\[16pt]
&\leq&\displaystyle c_1(K) \lambda^{\frac{n-6}{2n-8}} \frac{1}{M_{\lambda,+}^{2^*-1}}  \int_0^{{M_{\lambda,+}^{\beta}r_\lambda}} s^{-(n-2)}s^{n-1}\
\ ds = \displaystyle c_1(K) \lambda^{\frac{n-6}{2n-8}} \frac{1}{M_{\lambda,+}^{2^*-1}}  \int_0^{{M_{\lambda,+}^{\beta}r_\lambda}} s\ ds\\[16pt]
&=&\displaystyle c_2(K) \lambda^{\frac{n-6}{2n-8}} \frac{1}{M_{\lambda,+}^{2^*-1}} \ M_{\lambda,+}^{2\beta}r_\lambda^2 
=\displaystyle c_2(K) \lambda^{\frac{n-6}{2n-8}} \frac{1}{M_{\lambda,+}} r_\lambda^2 \rightarrow 0, \ \hbox{as}\ \lambda\rightarrow 0.
\end{array}
\end{eqnarray*}
Since this inequality is uniform respect to $x \in K$ we have  $\|I_{1,\lambda}\|_{\infty,K} \rightarrow 0$ as $\lambda \rightarrow 0$. The proof that $\|I_{3,\lambda}\|_{\infty,K} \rightarrow 0$ is quite similar to the previous one, in fact with small modifications we get the following uniform estimate:

\begin{eqnarray*}
\begin{array}{lll}
|I_{3,\lambda}|&\leq&\displaystyle  \lambda^{\frac{n-6}{2n-8}} \frac{1}{M_{\lambda,+}}\int_{B_{M_{\lambda,+}^{\beta}r_\lambda}} \left|G\left(x,\frac{y}{M_{\lambda,+}^{\beta}}\right)\right| [\tilde u_\lambda^+\left(y\right)]^{2^*-1} \ dy\\[16pt]
&\leq&\displaystyle c(K) \lambda^{\frac{n-6}{2n-8}} \frac{1}{M_{\lambda,+}}\int_{B_{M_{\lambda,+}^{\beta}r_\lambda}} \left\{1+  \frac{ 1}{n (n-2)}\left|{y}\right|^2\right\}^{-(n+2)/2} \ dy\\[16pt]
&\leq&\displaystyle c(K) \lambda^{\frac{n-6}{2n-8}} \frac{1}{M_{\lambda,+}}\int_{\R^n}\left\{1+  \frac{ 1}{n (n-2)}\left|{y}\right|^2\right\}^{-(n+2)/2}\ dy \\[16pt]
&=&\displaystyle c_1(K) \ \lambda^{\frac{n-6}{2n-8}} \frac{1}{M_{\lambda,+}},\ \ \hbox{as}\ \lambda\rightarrow 0.
\end{array}
\end{eqnarray*}

The proof for $I_{2,\lambda}$ is more delicate since for all small $\lambda>0$  the Green function is not bounded when $x \in K$, $y \in A_{r_\lambda}$. We split the Green function in the singular part and the regular part so that
$$ I_{2,\lambda}= \lambda^{\frac{n-6}{2n-8}} \int_{A_{r_\lambda}} G_{sing}(x,y) u_\lambda^-(y) \ dy+\lambda^{\frac{n-6}{2n-8}} \int_{A_{r_\lambda}} G_{reg}(x,y) u_\lambda^-(y) \ dy.$$
The singular part of the Green function is given by $\frac{1}{n(2-n)\omega_n} \frac{1}{|x-y|^{n-2}}$, we want to show that 
$$ \lambda^{\frac{n-6}{2n-8}} \frac{1}{n(2-n)\omega_n} \int_{A_{r_\lambda}} \frac{1}{|x-y|^{n-2}} u_\lambda^-(y) \ dy  \rightarrow 0$$
uniformly for $x \in K$. The usual change of variable gives

\begin{eqnarray*}
\begin{array}{lll}
&&\displaystyle  \lambda^{\frac{n-6}{2n-8}} \frac{1}{n(2-n)\omega_n} \int_{A_{r_\lambda}} \frac{1}{|x-y|^{n-2}} u_\lambda^-(y) \ dy\\
&=& \displaystyle \frac{\lambda^{\frac{n-6}{2n-8}}}{M_{\lambda,-}^{2^*}} \frac{1}{n(2-n)\omega_n}\int_{\tilde A_{r_\lambda}} \frac{1}{|x-\frac{w}{M_{\lambda,-}^{\beta}}|^{n-2}} u_\lambda^-\left(\frac{w}{M_{\lambda,-}^{\beta}}\right) \ dw.
\end{array}
\end{eqnarray*}
Let  $\eta$ be a positive real number such that $\eta < \min\{\frac{d(0,K)}{2};\frac{d(K,\partial B_1)}{2}\}$, where $d(\cdot,\cdot)$ denotes the euclidean distance. It's clear that for all $\lambda>0$ sufficiently small we have $B(x,\eta) \subset \subset A_{r_\lambda}$, for all $x \in K$. Thus $B(M_{\lambda,-}^{\beta}x,M_{\lambda,-}^{\beta}\eta) \subset \subset \tilde A_{r_\lambda}$,  for all $x \in K$, and we split the last integral in two parts as indicated below:
\begin{eqnarray*}
\begin{array}{lll}
&& \displaystyle \frac{\lambda^{\frac{n-6}{2n-8}}}{M_{\lambda,-}^{2^*}} \frac{1}{n(2-n)\omega_n} \int_{\tilde A_{r_\lambda}} \frac{1}{|x-\frac{w}{M_{\lambda,-}^{\beta}}|^{n-2}} u_\lambda^-\left(\frac{w}{M_{\lambda,-}^{\beta}}\right) \ dw \\[16pt]
&=& \displaystyle \frac{\lambda^{\frac{n-6}{2n-8}}}{M_{\lambda,-}^{2^*}} \frac{1}{n(2-n)\omega_n} \int_{|M_{\lambda,-}^{\beta}x- w|<M_{\lambda,-}^{\beta}\eta} \frac{1}{|x-\frac{w}{M_{\lambda,-}^{\beta}}|^{n-2}} u_\lambda^-\left(\frac{w}{M_{\lambda,-}^{\beta}}\right) \ dw \\[16pt]
 &+& \displaystyle  \frac{\lambda^{\frac{n-6}{2n-8}}}{M_{\lambda,-}^{2^*}} \frac{1}{n(2-n)\omega_n} \int_{\{|M_{\lambda,-}^\beta x - w|\geq M_{\lambda,-}^\beta\eta \}\ \cap\  \tilde A_{r_\lambda}} \frac{1}{|x-\frac{w}{M_{\lambda,-}^{\beta}}|^{n-2}} u_\lambda^-\left(\frac{w}{M_{\lambda,-}^{\beta}}\right) \ dw\\[16pt]
 &=& \displaystyle \frac{\lambda^{\frac{n-6}{2n-8}}}{M_{\lambda,-}^{2^*-1}} \frac{1}{n(2-n)\omega_n} \int_{|M_{\lambda,-}^{\beta}x- w|<M_{\lambda,-}^{\beta}\eta} \frac{{M_{\lambda,-}^{(n-2)\beta}}}{|{M_{\lambda,-}^{\beta}}x-{w}|^{n-2}} \tilde u_\lambda^-\left(w \right) \ dw \\[16pt]
 &+& \displaystyle  \frac{\lambda^{\frac{n-6}{2n-8}}}{M_{\lambda,-}^{2^*-1}} \frac{1}{n(2-n)\omega_n} \int_{\{|M_{\lambda,-}^\beta x - w|\geq M_{\lambda,-}^\beta\eta \}\ \cap\  \tilde A_{r_\lambda}} \frac{{M_{\lambda,-}^{(n-2)\beta}}}{|{M_{\lambda,-}^{\beta}}x-{w}|^{n-2}} \tilde u_\lambda^-\left(w \right) \ dw
 := \tilde I_{A,\lambda}+\tilde I_{B,\lambda}.
\end{array}
\end{eqnarray*}

Let's show that $\tilde I_{A,\lambda} \rightarrow 0$, uniformly for $x \in K$, as $\lambda \rightarrow 0$. First, by making the change of variable  $z:=w-{M_{\lambda,-}^{\beta}}x$ we have
\begin{eqnarray*}
\tilde I_{A,\lambda}&=& \displaystyle \frac{\lambda^{\frac{n-6}{2n-8}}}{M_{\lambda,-}^{2^*-1}} \frac{1}{n(2-n)\omega_n} \int_{|z|<M_{\lambda,-}^{\beta}\eta} \frac{{M_{\lambda,-}^{(n-2)\beta}}}{|z|^{n-2}} \tilde u_\lambda^-\left(z+ {M_{\lambda,-}^{\beta}}x\right) \ dz.
\end{eqnarray*}
Let us fix $\epsilon \in (0,\frac{n-2}{2})$ and set $C=\frac{2}{n-2}\epsilon$. Thanks to the uniform upper bound (\ref{Geulbnp}), since
\begin{equation}\label{osservuup}
 |M_{\lambda,-}^{\beta}x+z|\geq |M_{\lambda,-}^{\beta}|x|-|z|| = M_{\lambda,-}^{\beta}|x|-|z|  \geq M_{\lambda,-}^{\beta} (|x|-\eta)> M_{\lambda,-}^{\beta} \frac{d(0,K)}{2}\geq  M_{\lambda,-}^{\beta} \eta,
 \end{equation}
for all $x \in K$, for all $z$ such that $|z|< \eta M_{\lambda,-}^{\beta} $, then for all sufficiently small $\lambda$ we have
\begin{eqnarray*}
\begin{array}{lll}
|\tilde I_{A,\lambda}|&\leq&  \displaystyle \frac{\lambda^{\frac{n-6}{2n-8}}}{M_{\lambda,-}^{2^*-1}}  \frac{1}{n(n-2)\omega_n}\int_{|z|<M_{\lambda,-}^{\beta}\eta} \frac{{M_{\lambda,-}^{(n-2)\beta}}}{|z|^{n-2}}  \left[1+\frac{1}{n (n-2)} C |z+ {M_{\lambda,-}^{\beta}}x|^2\right]^{-(n-2)/2}  \ dz\\[16pt]
&\leq&  \displaystyle \frac{\lambda^{\frac{n-6}{2n-8}}}{M_{\lambda,-}^{2^*-1}}  c_1 \int_{|z|<M_{\lambda,-}^{\beta}\eta} \frac{{M_{\lambda,-}^{(n-2)\beta}}}{|z|^{n-2}}  \left[M_{\lambda,-}^{2\beta} \eta^2\right]^{-(n-2)/2}  \ dz\\[16pt]
&=&  \displaystyle \frac{\lambda^{\frac{n-6}{2n-8}}}{M_{\lambda,-}^{2^*-1}} c_2(K)\omega_n \int_{0}^{M_{\lambda,-}^{\beta}\eta} r  \ dr
=  \displaystyle \frac{\lambda^{\frac{n-6}{2n-8}}}{M_{\lambda,-}^{2^*-1}} c_2(K) \omega_n \frac{M_{\lambda,-}^{2\beta}\eta^2}{2}\\[16pt]
&=&  \displaystyle c_3(K)\frac{\lambda^{\frac{n-6}{2n-8}}}{M_{\lambda,-}} \rightarrow 0, \ \ \hbox{as} \ \lambda \rightarrow 0. 
\end{array}
\end{eqnarray*}
Thus $\tilde I_{A,\lambda} \rightarrow 0$, uniformly for $x \in K$, as $\lambda \rightarrow 0$.
Now we prove that the same holds for $\tilde I_{B,\lambda}$. 

\begin{eqnarray*}
\begin{array}{lll}
|\tilde I_{B,\lambda}| &\leq&\displaystyle  \frac{\lambda^{\frac{n-6}{2n-8}}}{M_{\lambda,-}^{2^*-1}} \frac{1}{n(n-2)\omega_n} \int_{\{|M_{\lambda,-}^\beta x - w|\geq M_{\lambda,-}^\beta\eta \}\ \cap\  \tilde A_{r_\lambda}} \frac{1}{|\eta|^{n-2}} \tilde u_\lambda^-\left(w \right) \ dw\\[16pt]
  &\leq&\displaystyle  \frac{\lambda^{\frac{n-6}{2n-8}}}{M_{\lambda,-}^{2^*-1}} c(K)\int_{\tilde A_{r_\lambda}} \tilde u_\lambda^-\left(w \right) \ dw\\[16pt]
    &\leq&\displaystyle  \frac{\lambda^{\frac{n-6}{2n-8}}}{M_{\lambda,-}^{2^*-1}} c(K)\int_{|w|\leq h} 1\ dw + \frac{\lambda^{\frac{n-6}{2n-8}}}{M_{\lambda,-}^{2^*-1}} c(K)\int_{h <|w|<M_{\lambda,-}^{\beta}} \left[1+\frac{1}{n (n-2)} C|w|^2\right]^{-(n-2)/2} \ dw \\[16pt]
   &\leq&\displaystyle  \frac{\lambda^{\frac{n-6}{2n-8}}}{M_{\lambda,-}^{2^*-1}} c_1(K)+ \frac{\lambda^{\frac{n-6}{2n-8}}}{M_{\lambda,-}^{2^*-1}} c_2(K) \int_{h}^{M_{\lambda,-}^{\beta}}  r \ dr\\[16pt]
    &=&\displaystyle  \frac{\lambda^{\frac{n-6}{2n-8}}}{M_{\lambda,-}^{2^*-1}} c_1(K)+ \frac{\lambda^{\frac{n-6}{2n-8}}}{M_{\lambda,-}^{2^*-1}} c_2(K) \left(\frac{{M_{\lambda,-}^{2\beta}}}{2}-\frac{h^2}{2}\right)\rightarrow 0, \ \hbox{as} \ \lambda\rightarrow 0,
\end{array}
\end{eqnarray*}
having used again (\ref{Geulbnp}).
Since this estimate is uniform for $x \in K$ we have proved that $\tilde I_{B,\lambda} \rightarrow 0$ in $C^0(K)$ and from this and the analogous result for $\tilde I_{A,\lambda}$ we have $  \lambda^{\frac{n-6}{2n-8}} \int_{A_{r_\lambda}} G_{sing}(x,y) u_\lambda^-(y) \ dy \rightarrow 0$ in $C^0(K)$. To complete the proof of $I_{2,\lambda} \rightarrow 0$ in $C^0(K)$ it remains to prove that $\lambda^{\frac{n-6}{2n-8}} \int_{A_{r_\lambda}} G_{reg}(x,y) u_\lambda^-(y) \ dy \rightarrow 0$ in $C^0(K)$. This is easy because the regular part of the Green function for the ball is uniformly bounded, to be precise let $l(K):= \sup\{d(0,x), x \in K\} $, clearly, being $K$ a compact subset of $B_1 -\{0\}$, we have $l(K)<1$ and since it is well known that
$$G_{reg}(x,y)= \frac{1}{n(2-n)\omega_n}\frac{1}{\left|(|x||y|)^2+1-2x\cdot y\right|^{\frac{n-2}{2}}},$$ 
we have for all $x \in K$, $y \in A_{r_\lambda}$
\begin{equation}\label{gallrigl}
\begin{array}{lllll}
 \displaystyle\frac{1}{\left|(|x||y|)^2+1-2x\cdot y\right|^{\frac{n-2}{2}}} 
 &\leq& \displaystyle\frac{1}{\left|(1-|x||y|)^2\right|^{\frac{n-2}{2}}}\\[16pt]
    &\leq& \displaystyle\frac{1}{\left|1-l(K)\right|^{n-2}}\ .
    \end{array}
\end{equation}
Thus we have
\begin{eqnarray*}
\begin{array}{lll}
\displaystyle\left|\lambda^{\frac{n-6}{2n-8}} \int_{A_{r_\lambda}} G_{reg}(x,y) u_\lambda^-(y) \ dy\right|&\leq&\displaystyle c(K)\lambda^{\frac{n-6}{2n-8}}  \int_{A_{r_\lambda}}  |u_\lambda^-(y)| \ dy\\[16pt]
&=&\displaystyle c(K) \frac{\lambda^{\frac{n-6}{2n-8}}}{M_{\lambda,-}^{2^*}}  \int_{\tilde A_{r_\lambda}}  \left|u_\lambda^-\left(\frac{w}{M_{\lambda,-}^{\beta}}\right)\right| \ dw\\[16pt]
&=&\displaystyle c(K) \frac{\lambda^{\frac{n-6}{2n-8}}}{M_{\lambda,-}^{2^*-1}}  \int_{\tilde A_{r_\lambda}}  \left|\tilde u_\lambda^-(w)\right| \ dw.
\end{array}
\end{eqnarray*}
As in the previous case we see that $\displaystyle c(K) \frac{\lambda^{\frac{n-6}{2n-8}}}{M_{\lambda,-}^{2^*-1}}  \int_{\tilde A_{r_\lambda}}  \left|\tilde u_\lambda^-(w)\right| \ dw\rightarrow 0$ and the proof of $I_{2,\lambda} \rightarrow 0$ in $C^0(K)$ is complete.

Now to end the proof we need to show that $I_{4,\lambda} \rightarrow \tilde c(n) G(x,0)$ in $C^0(K)$. We start making the usual change of variable
\begin{eqnarray*}
I_{4,\lambda}= \displaystyle   \lambda^{-\frac{n-2}{2n-8}}\frac{1}{M_{\lambda,-}} \int_{\tilde A_{r_\lambda}} G\left(x,\frac{w}{M_{\lambda,-}^\beta}\right)  [\tilde u_\lambda^-(w)]^{2^*-1} \ dw.\\[16pt]
\end{eqnarray*}
We split the Green function in the singular and the regular part, so that
\begin{eqnarray*}
\begin{array}{lllll}
\displaystyle I_{4,\lambda} &=&  \displaystyle \frac{1}{n(2-n)\omega_n}   \frac{\lambda^{-\frac{n-2}{2n-8}}}{M_{\lambda,-}}  \int_{\tilde A_{r_\lambda}} \frac{1}{|x- \frac{w}{M_{\lambda,-}^\beta}|^{n-2}}  [\tilde u_\lambda^-(w)]^{2^*-1} \ dw\\[18pt]
 &+& \displaystyle   \frac{\lambda^{-\frac{n-2}{2n-8}}}{M_{\lambda,-}} \int_{\tilde A_{r_\lambda}} G_{reg}\left(x,\frac{w}{M_{\lambda,-}^\beta}\right)  [\tilde u_\lambda^-(w)]^{2^*-1} \ dw 
\end{array}
\end{eqnarray*}
We begin with the singular integral which is more delicate. We want to show that 
\begin{equation}\label{bowie}
\displaystyle   \frac{\lambda^{-\frac{n-2}{2n-8}}}{M_{\lambda,-}} \frac{1}{n(2-n)\omega_n} \int_{\tilde A_{r_\lambda}} \frac{1}{|x- \frac{w}{M_{\lambda,-}^\beta}|^{n-2}}  [\tilde u_\lambda^-(w)]^{2^*-1} \ dw \rightarrow \tilde c(n) G_{sing}(x,0) \ \ \hbox{in} \ C^0(K).
\end{equation}
As in the previous case we consider the ball $B(M_{\lambda,-}^\beta x,M_{\lambda,-}^\beta\eta) \subset \subset \tilde A_{r_\lambda}$, where $\eta>0$ is the same as before. Thus we have
\begin{eqnarray*}
\begin{array}{lllll}
 &&\displaystyle   \frac{\lambda^{-\frac{n-2}{2n-8}}}{M_{\lambda,-}} \frac{1}{n(2-n)\omega_n} \int_{\tilde A_{r_\lambda}} \frac{1}{|x- \frac{w}{M_{\lambda,-}^\beta}|^{n-2}}  [\tilde u_\lambda^-(w)]^{2^*-1} \ dw\\[16pt]
 &=&\displaystyle   \frac{\lambda^{-\frac{n-2}{2n-8}}}{M_{\lambda,-}} \frac{1}{n(2-n)\omega_n} \int_{|M_{\lambda,-}^\beta x - w|<M_{\lambda,-}^\beta\eta} \frac{M_{\lambda,-}^{(n-2)\beta}}{|M_{\lambda,-}^\beta x- w|^{n-2}}  [\tilde u_\lambda^-(w)]^{2^*-1} \ dw\\[16pt]
  &+&\displaystyle   \frac{\lambda^{-\frac{n-2}{2n-8}}}{M_{\lambda,-}} \frac{1}{n(2-n)\omega_n} \int_{\{|M_{\lambda,-}^\beta x - w|\geq M_{\lambda,-}^\beta\eta \}\ \cap\  \tilde A_{r_\lambda}} \frac{M_{\lambda,-}^{(n-2)\beta}}{|M_{\lambda,-}^\beta x- w|^{n-2}}  [\tilde u_\lambda^-(w)]^{2^*-1} \ dw\\[16pt]
  &:=& \tilde I_{C,\lambda} + \tilde I_{D,\lambda}.
\end{array}
\end{eqnarray*}
We show that $ \tilde I_{C,\lambda} \rightarrow 0$ in $C^0(K)$. As before, using the uniform upper bound (\ref{Geulbnp}) and (\ref{osservuup}) we get

\begin{eqnarray*}
\begin{array}{lll}
|\tilde I_{C,\lambda}|&=&  \displaystyle \frac{\lambda^{-\frac{n-2}{2n-8}}}{M_{\lambda,-}}  \frac{1}{n(n-2)\omega_n} \int_{|z|<M_{\lambda,-}^{\beta}\eta} \frac{{M_{\lambda,-}^{(n-2)\beta}}}{|z|^{n-2}} \left[\tilde u_\lambda^-\left(z+ {M_{\lambda,-}^{\beta}}x\right)\right]^{2^*-1} \ dz\\[16pt]
&\leq&  \displaystyle \frac{\lambda^{-\frac{n-2}{2n-8}}}{M_{\lambda,-}}  \frac{1}{n(n-2)\omega_n} \int_{|z|<M_{\lambda,-}^{\beta}\eta} \frac{{M_{\lambda,-}^{(n-2)\beta}}}{|z|^{n-2}}  \left[1+\frac{1}{n (n-2)} C |z+ {M_{\lambda,-}^{\beta}}x|^2\right]^{-(n+2)/2}  \ dz\\[16pt]
&\leq&  \displaystyle \frac{\lambda^{-\frac{n-2}{2n-8}}}{M_{\lambda,-}} c_1\int_{|z|<M_{\lambda,-}^{\beta}\eta} \frac{{M_{\lambda,-}^{(n-2)\beta}}}{|z|^{n-2}}  \left[M_{\lambda,-}^{2\beta} \eta^2 \right]^{-(n+2)/2}  \ dz\\[16pt]
&=&  \displaystyle \frac{\lambda^{-\frac{n-2}{2n-8}}}{M_{\lambda,-}} c_2(K) \int_{0}^{M_{\lambda,-}^{\beta}\eta} \frac{{M_{\lambda,-}^{(n-2)\beta}}}{r^{n-2}}  M_{\lambda,-}^{-(n+2)\beta}  r^{n-1}  \ dr\\[16pt]
&=&  \displaystyle \frac{\lambda^{-\frac{n-2}{2n-8}}}{M_{\lambda,-}} c_2(K)\frac{1}{M_{\lambda,-}^{4\beta}} \int_{0}^{M_{\lambda,-}^{\beta}\eta} r  \ dr
=  \displaystyle \frac{\lambda^{-\frac{n-2}{2n-8}}}{M_{\lambda,-}}  c_2(K) \frac{1}{M_{\lambda,-}^{4\beta}} \frac{M_{\lambda,-}^{2\beta}\eta^2}{2}\\[16pt]
&=&  \displaystyle c_3(K)\frac{\lambda^{-\frac{n-2}{2n-8}}}{M_{\lambda,-}}\frac{1}{M_{\lambda,-}^{2\beta}}.
\end{array}
\end{eqnarray*}
Since $\frac{\lambda^{-\frac{n-2}{2n-8}}}{M_{\lambda,-}}$ is bounded (see Proposition \ref{asestprop2} (iv) and Remark \ref{ovvio}) then $\tilde I_{C,\lambda} \rightarrow 0$ uniformly for $x \in K$.
Now we show that $\tilde I_{D,\lambda} \rightarrow \tilde c(n) G_{sing}(x,0)$ in $C^0(K)$. We have
\begin{eqnarray*}
%\begin{array}{lllll}
 %&&\displaystyle   \frac{\lambda^{-\frac{n-2}{2n-8}}}{M_{\lambda,-}}  \frac{1}{n(2-n)\omega_n}  \int_{\{|M_{\lambda,-}^\beta x - w|\geq M_{\lambda,-}^\beta\eta \}\ \cap\  \tilde A_{r_\lambda}} \frac{M_{\lambda,-}^{(n-2)\beta}}{|M_{\lambda,-}^\beta x- w|^{n-2}}  [\tilde u_\lambda^-(w)]^{2^*-1} \ dw\\[16pt]
  \tilde I_{D,\lambda} &=&\displaystyle   \frac{\lambda^{-\frac{n-2}{2n-8}}}{M_{\lambda,-}}  \frac{1}{n(2-n)\omega_n}  \int_{\{|x - \frac{w}{M_{\lambda,-}^\beta }|\geq \eta \}\ \cap\  \tilde A_{r_\lambda}} \frac{1}{|x- \frac{w}{M_{\lambda,-}^\beta }|^{n-2}}  [\tilde u_\lambda^-(w)]^{2^*-1} \ dw
%\end{array}
\end{eqnarray*}

The first step is to prove that for all $w \in \R^n-\{0\}$
 \begin{equation}\label{convpuntgreen}
\chi(w)_{\left\{\{|x - \frac{w}{M_{\lambda,-}^\beta }|\geq \eta \}\ \cap\  \tilde A_{r_\lambda}\right\}}  \frac{1}{n(2-n)\omega_n}  \frac{1}{|x- \frac{w}{M_{\lambda,-}^\beta }|^{n-2}}  [\tilde u_\lambda^-(w)]^{2^*-1} \rightarrow  G_{sing}(x,0) \delta_{0,\mu}^{2^*-1}(w),
\end{equation}
 uniformly for $x \in K$. First, observe that we need only to show that
\begin{equation}\label{caprasgarbi}
  \frac{1}{|x- \frac{w}{M_{\lambda,-}^\beta }|^{n-2}}  [\tilde u_\lambda^-(w)]^{2^*-1} \rightarrow   \frac{1}{|x|^{n-2}} \delta_{0,\mu}^{2^*-1}(w)\ \ \ \hbox{in}\ C^0(K). 
  \end{equation}
 In fact if we fix $w \in  \R^n-\{0\} $, and $\lambda>0$ is sufficiently small so that $w \in \tilde A_{r_\lambda}$ and $\frac{w}{M_{\lambda,-}^\beta} < \frac{d(0,K)}{2}$  then we have $|x - \frac{w}{M_{\lambda,-}^\beta }|\geq \eta$, for all $x \in K$. Hence we get
$$ \left|\chi(w)_{\left\{\{|x - \frac{w}{M_{\lambda,-}^\beta }|\geq \eta \}\ \cap\  \tilde A_{r_\lambda}\right\}}-1\right|=\chi(w)_{\left\{{\{|x - \frac{w}{M_{\lambda,-}^\beta }|< \eta \} \cup  \tilde A_{r_\lambda}^c }\right\}}=0,$$
for all $x \in K$, for all $\lambda>0$ sufficiently small, from which we deduce that  $$\chi(w)_{\left\{\{|x - \frac{w}{M_{\lambda,-}^\beta }|\geq \eta \}\ \cap\  \tilde A_{r_\lambda}\right\}} \rightarrow 1\  \ \ \hbox{in}\ C^0(K).$$ 

Now the proof of (\ref{caprasgarbi})  is trivial if we show that, for any fixed $w \in \R^{n}-\{0\}$
\begin{equation}\label{sgarbi}
\left|  \frac{1}{\left|x- \frac{w}{M_{\lambda,-}^\beta}\right|^{n-2}} -  \frac{1}{|x|^{n-2}} \right|\leq c(K) \left| \frac{w}{M_{\lambda,-}^\beta}\right|
\end{equation}
for all $x \in K$ and for all $\lambda>0$ sufficiently small.  
This is an elementary computation but for the sake of completeness we give the proof.  %Since  $K$ is a compact subset of $B_1- \{0\}$ and $\left| x- \frac{w}{M_{\lambda,-}^\beta}\right|\geq \frac{d(0,K)}{2}$ for all $x \in K$, for all $\lambda>0$ sufficiently small, we have that both $x$ and $x- \frac{w}{M_{\lambda,-}^\beta}$ stay away from the origin. In particular 
We observe that the segment  $\sigma\left(x,x-\frac{w}{M_{\lambda,-}^\beta}\right)$ joining $x$ and $x-\frac{w}{M_{\lambda,-}^\beta}$  is an uniformly bounded set and stays  away from the origin. In fact for all $x \in K$, $t \in [0,1]$ and for all $\lambda >0$ sufficiently small we have
\begin{equation}\label{sgarbi1}
 \left|x-t\frac{w}{M_{\lambda,-}^\beta}  \right| \leq |x| + |t|\left|\frac{w}{M_{\lambda,-}^\beta}\right| < 1 + \frac{d(0,K)}{2}
\end{equation}
\begin{equation}\label{sgarbi2}
\left|x-t\frac{w}{M_{\lambda,-}^\beta}  \right|\geq \left||x|-|t|\frac{|w|}{M_{\lambda,-}^\beta}  \right|\geq d(0,K) - t \frac{d(0,K)}{2}\geq\frac{d(0,K)}{2}.
\end{equation}

Thus, setting $g(x):=\frac{1}{|x|^{n-2}}$, by Lagrange's theorem we have $$g(x- \frac{w}{M_{\lambda,-}^\beta})-g(x)=\nabla g (\xi_{\lambda,x})\cdot\frac{w}{M_{\lambda,-}^\beta},$$ where $\xi_{\lambda,x}$ lies on  $\sigma\left(x,x-\frac{w}{M_{\lambda,-}^\beta}\right)$. By (\ref{sgarbi1}) and  (\ref{sgarbi2}) we deduce that $|\nabla g (\xi_{\lambda,x})|$ is uniformly bounded \footnote{by $\|\nabla g\|_{\infty,R(K)}$, where $R(K)$ is the compact annulus  $R(K):=\{x \in \R^n; \ \frac{d(0,K)}{2}\leq |x| \leq 1 + \frac{d(0,K)}{2}
 \}$} and (\ref{sgarbi}) is proved.

To complete the first part of the proof we apply Lebesgue's theorem.  For all $x \in K$, $w \in \R^n-\{0\}$ we have 

\begin{eqnarray*}
\begin{array}{lll}
&&\displaystyle \left|\chi_{\left\{\{|x - \frac{w}{M_{\lambda,-}^\beta }|\geq \eta \}\ \cap\  \tilde A_{r_\lambda}\right\}} \frac{1}{n(2-n)\omega_n} \frac{1}{|x- \frac{w}{M_{\lambda,-}^\beta }|^{n-2}}  [\tilde u_\lambda^-(w)]^{2^*-1} \right|\\[24pt]
 &\leq& \displaystyle\eta^{-(n-2)} \frac{1}{n(n-2)\omega_n} \frac{1}{|x|^{n-2}}  [U_{h}(w)]^{2^*-1}\\[16pt]
&=&\displaystyle c_1(K) [U_{h}(w)]^{2^*-1},
\end{array}
\end{eqnarray*}
where $U_{h}$ is the function defined in (\ref{uupperbound}). Since $(U_{h})^{2^*-1} \in L^1(\R^n)$ and thanks to (\ref{convpuntgreen}), (iv) of Proposition \ref{asestprop2},  by Lebesgue's theorem we deduce (\ref{bowie}), where $G_{sing}(x,0)= \frac{1}{n(2-n)\omega_n}\frac{1}{|x|^{n-2}}$, $\tilde c(n)= (\lim_{\lambda \rightarrow 0}\frac{\lambda^{-\frac{n-2}{2n-8}}}{M_{\lambda,-}}) \int_{\R^n}\delta_{0,\mu}^{2^*-1} (w) \ dw$. It's an elementary computation to see that $\tilde c(n)$ equals the expected constant $\omega_n \frac{c_2(n)^{\frac{n-2}{2n-8}}}{c_1(n)^{\frac{4}{2n-8}}}$, where $c_1(n), c_2(n)$ are the constants defined in  Proposition \ref{asestprop1}. And the proof of (\ref{bowie}) is done.

Finally we prove that  

\begin{equation}\label{bowie2}
\displaystyle   \frac{\lambda^{-\frac{n-2}{2n-8}}}{M_{\lambda,-}} \int_{\tilde A_{r_\lambda}} G_{reg}\left(x,\frac{w}{M_{\lambda,-}^\beta}\right)  [\tilde u_\lambda^-(w)]^{2^*-1} \ dw \rightarrow \tilde c(n) G_{reg}(x,0) \ \ \hbox{in} \ C^0(K).
\end{equation}
Since $$G_{reg}\left(x,\frac{w}{M_{\lambda,-}^\beta}\right)= \frac{1}{n(2-n)\omega_n} \frac{1}{\left||x|^2\frac{|w|^2}{M_{\lambda,-}^{2\beta}}+1-2x\cdot \frac{w}{M_{\lambda,-}^\beta}\right|^{\frac{n-2}{2}}}$$  by Lagrange theorem, repeating a similar argument as in the proof of (\ref{sgarbi}),  we deduce that for any fixed $w \in \R^n-\{0\}$
$$G_{reg}\left(x,\frac{w}{M_{\lambda,-}^\beta}\right) \rightarrow G_{reg}(x,0) \ \ \ \hbox{in} \ C^0(K). $$
Thus for any $w \in \R^n-\{0\}$ we have  $$G_{reg}\left(x,\frac{w}{M_{\lambda,-}^\beta}\right)  [\tilde u_\lambda^-(w)]^{2^*-1} \rightarrow  G_{reg}(x,0) \delta_{0,\mu}^{2^*-1}(w) \ \ \hbox{in} \ C^0(K).$$
Thanks to (\ref{gallrigl}) we know that $G_{reg}\left(x,\frac{w}{M_{\lambda,-}^\beta}\right)$ is uniformly bounded, moreover, as we have done in the proof of (\ref{bowie}), thanks to the upper bound (\ref{Geulbnp}),  Proposition \ref{asestprop2}  we deduce (\ref{bowie2}). 

To prove the local uniform convergence of  $\lambda^{-\frac{n-2}{2n-8}} \nabla u_\lambda$ to $\tilde c(n) \nabla G(x,0)$ we simply derive (\ref{formrappres}) and repeat the previous proof, taking into account that for $i=1,\ldots,n$ we have
$$\partial_{x_i}G_{sing}(x,y)=\frac{1}{n \omega_n} \frac{x_i-y_i}{|x-y|^{n}}.$$
\end{proof}

\end{document}